\documentclass{article}

\usepackage[english]{babel}

\usepackage[letterpaper,top=2cm,bottom=2cm,left=3cm,right=3cm,marginparwidth=1.75cm]{geometry}

\usepackage{amsmath}
\usepackage{amssymb}
\usepackage{amsthm}

\usepackage{tikz}
\usetikzlibrary{calc,intersections,arrows.meta}

\usepackage[normalem]{ulem}

\usepackage{graphicx}
\usepackage[colorlinks=true, allcolors=blue]{hyperref}

\usepackage{xcolor}
\usepackage{subcaption}

\usepackage{array}
\usepackage{makecell}

\usepackage{subcaption}
\theoremstyle{plain}
\newtheorem{theorem}{Theorem}
\newtheorem{lemma}[theorem]{Lemma}
\newtheorem{proposition}[theorem]{Proposition}

\theoremstyle{definition}

\newtheorem{problem}[theorem]{Problem}
\newtheorem{example}[theorem]{Example}

\newcommand{\llangle}{\langle\hspace{-2.5pt}\langle}
\newcommand{\rrangle}{\rangle\hspace{-2.5pt}\rangle}

\DeclareMathOperator*{\esssup}{ess\,sup}
\DeclareMathOperator*{\argmax}{arg\,max}

\title{Approximation theorems in bilipschitz invariant theory}
\author{Jameson Cahill\thanks{Department of Mathematics and Statistics, University of North Carolina Wilmington, Wilmington, NC} \and Joseph W.\ Iverson\thanks{Department of Mathematics, Iowa State University, Ames, IA} \and Dustin G.\ Mixon\thanks{Department of Mathematics, The Ohio State University, Columbus, OH} \thanks{Translational Data Analytics Institute, The Ohio State University, Columbus, OH} \and Nathan Willey\footnotemark[3]}
\date{}

\begin{document}
\maketitle

\begin{abstract}
Bilipschitz invariant theory concerns low-distortion embeddings of orbit spaces into Euclidean space.
To date, embeddings with the smallest-possible distortion are known for only a few cases, to include: (a) planar rotations, (b) real phase retrieval, and (c) finite reflection groups.
Here, we prove that for all three of these cases, the smallest possible distortion is nearly achieved by a composition of a ``max filter bank'' with a linear transformation.
Our proof amounts to a two-step process: first, we show it suffices to demonstrate a certain inclusion of Lipschitz function spaces, and second, we prove that inclusion, using fundamentally different approaches for the three cases.
We also show that these cases interact differently with a few related function spaces, which suggests that a unified treatment would be nontrivial.
\end{abstract}

\section{Introduction}

\subsection{Bilipschitz invariant theory}

Consider a finite group $G\leq\operatorname{O}(d)$.
Its natural action partitions the Euclidean space $\mathbb{R}^d$ into orbits $[x]:=G\cdot x$, which we view as points in the orbit space
\[
\mathbb{R}^d/G
:=\{[x]:x\in\mathbb{R}^d\}.
\]
Endowing $\mathbb{R}^d/G$ with the quotient metric
\[
d_{\mathbb{R}^d/G}([x],[y])
=\min_{g\in G}\|x-gy\|,
\]
we seek a $G$-invariant map $f\colon\mathbb{R}^d\to\mathbb{R}^n$ with the property that the induced map $f^\downarrow\colon\mathbb{R}^d/G\to\mathbb{R}^n$ minimally distorts the metric.

We are motivated by applications in group-invariant signal processing~\cite{BendoryBMZS:17,BandeiraBSKNWPW:23} and machine learning~\cite{BlumSV:23}, where a metric-preserving feature map into Euclidean space grants access to standard Euclidean tools; see for example Section~2 in~\cite{CahillIM:24}.
The result is a burgeoning theory of \textit{bilipschitz invariants}~\cite{ErikssonB:18,Derkson:24,DymG:25,DymWTLB:25}, much of which takes inspiration from ideas in phase retrieval~\cite{FreemanOPT:24,AmirBDE:26}.

Explicitly, we say $f^\downarrow$ is \textbf{bilipschitz} if the constants
\[
\alpha(f^\downarrow)
:=\inf_{\substack{x,y\in\mathbb{R}^d\\ [x]\neq [y]}}\frac{\|f(x)-f(y)\|}{d_{\mathbb{R}^d/G}([x],[y])},
\qquad
\beta(f^\downarrow)
:=\sup_{\substack{x,y\in\mathbb{R}^d\\ [x]\neq [y]}}\frac{\|f(x)-f(y)\|}{d_{\mathbb{R}^d/G}([x],[y])}
\]
reside in $(0,\infty)$.
We want invariants for which the \textbf{distortion}
\[
\operatorname{dist}(f^\downarrow)
:=\frac{\beta(f^{\downarrow})}{\alpha(f^\downarrow)}
\]
is as small as possible, where we interpret $\operatorname{dist}(f^\downarrow)=\infty$ when $\alpha(f^\downarrow)=0$ or $\beta(f^\downarrow)=\infty$.
Observe that the distortion is always at least $1$, with equality precisely when $f^\downarrow$ is an isometric embedding of $\mathbb{R}^d/G$ into $\mathbb{R}^n$ times a positive scalar, namely, $\alpha(f^\downarrow)=\beta(f^\downarrow)$.
If we also consider $G$-invariant maps into infinite-dimensional Hilbert spaces, then the best-possible distortion is known as the \textbf{Euclidean distortion} of the orbit space $\mathbb{R}^d/G$:
\[
c_2(\mathbb{R}^d/G)
:=\inf_{\substack{\text{$G$-inv }f\colon\mathbb{R}^d\to H\\\text{Hilbert space $H$}}}\operatorname{dist}(f^\downarrow).
\]
In certain special cases, the exact value of the Euclidean distortion $c_2(\mathbb{R}^d/G)$ is known, and it is even known how to achieve this distortion with a finite-dimensional target space $\mathbb{R}^n$.
What follows are some important instances of this phenomenon.

\begin{example}\
\label{ex.optimal coordinates}
\begin{itemize}
\item[(a)]
\textit{Planar rotations.}
Suppose $d=2$, and let $G$ be a finite subgroup of $\operatorname{SO}(2)$ of order $r>1$.
Then $c_2(\mathbb{R}^2/G)=r\sin\frac{\pi}{2r}$.
Furthermore, if we define $h\colon\mathbb{R}^2\to\mathbb{R}\times\mathbb{C}\cong\mathbb{R}^3$ by $h(0)=0$ and
\[
h(x)
=\|x\|\cdot\big(\,
\cos\tfrac{\pi}{2r},\,\,
(\tfrac{x_1+\mathrm{i}x_2}{\|x\|})^r \sin\tfrac{\pi}{2r}\,\big)
\]
for $x\neq0$, then $h$ is $G$-invariant and $\operatorname{dist}(h^\downarrow)=c_2(\mathbb{R}^2/G)$.
(This is perhaps more intuitive after identifying $\mathbb{R}^2$ with $\mathbb{C}$, since then the second coordinate corresponds to raising the phase of the input to the $r$th power.)
This was first established by Corollary~38 and Example~18 in~\cite{CahillIM:24}, and then further generalized by Theorem~16 in~\cite{Blum-SmitEtal:25}.
\item[(b)]
\textit{Real phase retrieval.}
Suppose $G=\{\pm I\}$. 
Then $c_2(\mathbb{R}^d/G)=\sqrt{2}$.
Furthermore, if we define $h\colon\mathbb{R}^d\to\mathbb{R}^{d\times d}_{\operatorname{sym}}\cong\mathbb{R}^{\binom{d+1}{2}}$ in terms of the operator square root by
\[
h(x)
=\sqrt{xx^\top},
\]
then $h$ is $G$-invariant and $\operatorname{dist}(h^\downarrow)=c_2(\mathbb{R}^d/G)$.
This was first established by Corollary~36 and Example~16 in~\cite{CahillIM:24}, and then further generalized by Proposition~17 in~\cite{Blum-SmitEtal:25}.
\item[(c)]
\textit{Reflections.}
Suppose $G\leq\operatorname{O}(d)$ is a finite reflection group, that is, a finite subgroup of $\operatorname{O}(d)$ generated by reflections across hyperplanes through the origin of $\mathbb{R}^d$.
Then $c_2(\mathbb{R}^d/G)=1$.
Next, the hyperplanes that correspond to all reflections in $G$ divide $\mathbb{R}^d$ into polyhedral cones known as \textbf{Weyl chambers}.
If we fix a closed Weyl chamber $W\subseteq\mathbb{R}^d$ of $G$ and define $h\colon\mathbb{R}^d\to\mathbb{R}^d$ by
\[
h(x)\in [x]\cap W,
\]
then $h$ is $G$-invariant and $\operatorname{dist}(h^\downarrow)=c_2(\mathbb{R}^d/G)$.
This result should probably be considered folklore, but it appears, for example, in Lemma~8 of~\cite{MixonP:23}.
\end{itemize}
\end{example}

In all three cases above, the optimal bilipschitz embedding $h$ we describe enjoys the emergent property of \textbf{positive homogeneity}, meaning
\[
h(rx)
=rh(x)
\]
for every $x\in\mathbb{R}^d$ and $r\geq0$.
In fact, \textit{every} optimal bilipschitz embedding of $\mathbb{R}^d/G$ that is known to date exhibits this behavior (at least once we translate the target space so that $[0]\mapsto 0$); see~\cite{Blum-SmitEtal:25} for a recent survey of these results.
We do not yet have an explanation for this.

\subsection{Max filtering and main result}

In contrast with the case-specific work above, we seek a more universal approach to low-distortion Euclidean embeddings.
One kind of bilipschitz embedding known as max filtering~\cite{CahillIMP:25} gives partial progress in this direction.
Specifically, given a \textbf{template} $y\in\mathbb{R}^d$, consider the \textbf{max filter} $\mathbb{R}^d\to\mathbb{R}$ defined by
\[
x
\longmapsto\llangle [x],[y]\rrangle
:=\max_{g\in G}\langle x,gy\rangle.
\]
Then for $n=2d$ generic templates $y_1,\ldots,y_n\in\mathbb{R}^d$, it holds that
\[
f\colon x\longmapsto \left[\begin{array}{c}\llangle [x],[y_1]\rrangle\\\vdots\\ \llangle [x],[y_n]\rrangle\end{array}\right]
\]
determines a bilipschitz \textbf{max filter bank} $f^\downarrow\colon \mathbb{R}^d/G\to\mathbb{R}^n$~\cite{CahillIMP:25,BalanT:23}.
(In particular, the Euclidean distortion $c_2(\mathbb{R}^d/G)$ is finite for every finite subgroup $G \leq \operatorname{O}(d)$.)
There are explicit general bounds on the distortion of such embeddings for random templates~\cite{CahillIMP:25,MixonQ:25}, but these bounds are not believed to be tight, and in cases where $c_2(\mathbb{R}^d/G)$ is known, there tends to be a gap.
For example, in the case of Example~\ref{ex.optimal coordinates}(b), the main result in \cite{XiaXX:24} gives that the infimal distortion of max filter banks is 
\[
\sqrt{\tfrac{\pi}{\pi-2}}
\approx1.659
>1.414
\approx\sqrt{2}
=c_2(\mathbb{R}^d/G).
\]
The situation is even more extreme  in the case of Example~\ref{ex.optimal coordinates}(c), where $c_2(\mathbb{R}^d/G)=1$, but the best max filter banks have distortion as large as $\Theta(d)$, depending on $G$; see~\cite{MixonP:23}.

In some cases, this gap can be immediately closed by post-composing with a linear map.
The following example makes this explicit.

\begin{example}
\label{ex.sort is linear of max filter}
Suppose $G\leq\operatorname{O}(3)$ is the reflection group isomorphic to $S_3$ that acts on $\mathbb{R}^3$ by permuting coordinates.
Consider the templates
\[
y_1:=\left[\begin{array}{c}1\\0\\0\end{array}\right],
\qquad
y_2:=\left[\begin{array}{c}1\\1\\0\end{array}\right],
\qquad
y_3:=\left[\begin{array}{c}1\\1\\1\end{array}\right].
\]
Let $\operatorname{sort}\colon\mathbb{R}^3\to\mathbb{R}^3$ denote the function that rearranges the input coordinates in weakly decreasing order.
Then max filtering with the above templates is given by
\[
\left[\begin{array}{c}\llangle [x],[y_1]\rrangle\\\llangle [x],[y_2]\rrangle\\ \llangle [x],[y_3]\rrangle\end{array}\right]
=\left[\begin{array}{ccc}1&0&0\\1&1&0\\1&1&1\end{array}\right]\operatorname{sort}\left(\left[\begin{array}{c}x_1\\x_2\\x_3\end{array}\right]\right).
\]
Apply the inverse of the above $3\times 3$ matrix to both sides to get
\[
\operatorname{sort}\left(\left[\begin{array}{c}x_1\\x_2\\x_3\end{array}\right]\right)
=\left[\begin{array}{rrr}1&0&0\\-1&1&0\\0&-1&\phantom{-}1\end{array}\right]\left[\begin{array}{c}\llangle [x],[y_1]\rrangle\\\llangle [x],[y_2]\rrangle\\ \llangle [x],[y_3]\rrangle\end{array}\right].
\]
Notably, $[x]\mapsto\operatorname{sort}(x)$ is an instance of the optimal bilipschitz embedding described in Example~\ref{ex.optimal coordinates}(c), where $W$ consists of the vectors in $\mathbb{R}^3$ with coordinates listed in weakly decreasing order.
\end{example}

The key takeaway of Example~\ref{ex.sort is linear of max filter} is that, at least in this case:
\begin{center}
\textit{the optimal bilipschitz embedding is a linear transformation of a max filter bank.}
\end{center}
Perhaps surprisingly, there is a sense in which this seems to generalize.
Specifically, it appears as though the infimal distortion of linear transformations of max filter banks equals the Euclidean distortion, that is, the gap from max filter banks can be closed in general by post-composing with a carefully selected linear map.
We do not have a general proof of this claim, but what follows is our main result, which focuses on the three classes of groups from Example~\ref{ex.optimal coordinates}.
(Later, we provide numerics for other groups, too, which suggest that this result might hold in greater generality.)

\begin{theorem}[Main result]
\label{thm.main result}
Fix a group $G\leq\operatorname{O}(d)$ and a target dimension $n$ such that
\begin{itemize}
\item[(a)] 
$d=2$, $G$ is a nontrivial finite subgroup of $\operatorname{SO}(2)$, and $n=3$,
\item[(b)]
$G=\{\pm I\}$ and $n=\binom{d+1}{2}$, or
\item[(c)]
$G$ is a finite reflection group and $n=d$.
\end{itemize}
For each $\epsilon > 0$, there exists $m=m(\epsilon)$, along with a max filter bank $\Phi\colon \mathbb{R}^d/G\to\mathbb{R}^m$ and a linear map $L\colon\mathbb{R}^m\to\mathbb{R}^n$ such that the composition $L\circ\Phi$ has distortion at most $c_2(\mathbb{R}^d/G) + \epsilon$.
(In fact, in the case of (iii), one may take $\epsilon=0$ and $m=d$.)
\end{theorem}

Note that in each case of Theorem~\ref{thm.main result}, the value of $n$ matches the target dimension from the optimal embedding $h$ in Example~\ref{ex.optimal coordinates}, and in fact, our proof relies on approximating the coordinate functions of~$h$.
For the non-reflection group cases, the dimension $m$ of the ``hidden layer'' gets arbitrarily large as $\epsilon\to0$.
In particular, the final linear layer performs substantial dimensionality reduction.
We are not the first to post-compose an invariant with a linear map in order to decrease the target dimension~\cite{BalanT:23,BalanTW:24,BalanHS:25}, but our result is qualitatively different. 
Previous theorems establish that injectivity is maintained \textit{despite} the dimensionality reduction, i.e., ``It doesn't hurt.'' 
Here, we achieve near-optimal distortion \textit{because} of the dimensionality reduction: ``It helps!''

\subsection{Lipschitz function classes and proof of main result}

While Theorem~\ref{thm.main result} fundamentally concerns the Euclidean distortion of quotient spaces, our proof of this result amounts to an inclusion of certain spaces of Lipschitz functions, as enunciated in Theorem~\ref{thm.main result ii} below.
Before giving the connection between these theorems, we first review some basic notions involving Lipschitz continuity.

Let $X$ and $Y$ be metric spaces, where $X$ has cardinality at least $2$.
Given a function $f \colon X \to Y$, we define
\[
\alpha(f)
:=\inf_{\substack{x,y\in X\\x\neq y}}\frac{d_Y(f(x),f(y))}{d_X(x,y)},
\qquad
\beta(f)
:=\sup_{\substack{x,y\in X\\x\neq y}}\frac{d_Y(f(x),f(y))}{d_X(x,y)},
\qquad
\operatorname{dist}(f):=\frac{\beta(f)}{\alpha(f)},
\]
where as before we interpret $\operatorname{dist}(f) = \infty$ if $\alpha(f) = 0$ or $\beta(f) = \infty$.
We say $f$ is \textbf{Lipschitz} if $\beta(f) < \infty$.
For the specific case $Y = \mathbb{R}$, the Lipschitz functions $X \to \mathbb{R}$ form a vector space $\operatorname{Lip}(X)$.
Over this vector space, $\beta(\cdot)$ is a seminorm, which we alternatively denote by $\|\cdot\|_{\operatorname{Lip}}$ when we choose to think of it that way.
An application of the mean value theorem gives that the Lipschitz norm of a smooth function $\mathbb{R}^d\to\mathbb{R}$ equals the uniform norm of its gradient, and this phenomenon generalizes to Lipschitz functions (see Proposition~\ref{prop.rademacher} below).
Note that 
$\|f\|_{\operatorname{Lip}}=0$ if and only if $f$ is constant, and so $\|\cdot\|_{\operatorname{Lip}}$
induces a norm on the quotient of $\operatorname{Lip}(X)$ by the constant functions.
The \textbf{Lipschitz closure} of a set $\mathcal{F}\subseteq\operatorname{Lip}(X)$ consists of all functions that are arbitrarily close in $\|\cdot\|_{\operatorname{Lip}}$ to members of $\mathcal{F}$.
(In particular, it contains $f+c$ for every $f\in \mathcal{F}$ and every constant function $c$.)

By the following lemma (whose proof is contained in Section~\ref{sec.proof of key lemma}), we may conclude Theorem~\ref{thm.main result} once we show that the coordinate functions of an optimal bilipschitz embedding reside in the Lipschitz closure of the span of max filters.
In particular, proximity in Lipschitz norm implies proximity of distortions.
Importantly, this sort of behavior is not guaranteed by proximity in the uniform norm, which is the usual setting for universal approximation theorems.
(Also, the span of max filters does not generally enjoy a universal approximation result, even after restricting to $G$-invariant positively homogeneous Lipschitz functions; more on that later.)

\begin{lemma}
\label{lem.lipschitz closures suffice}
Fix a metric space $X$ and a group $G$ that acts on $X$ non-transitively, isometrically, and with topologically closed orbits.
Let $\mathcal{F}$ denote a set of $G$-invariant Lipschitz functions $X\to\mathbb{R}$, and suppose $h\colon X\to\mathbb{R}^n$ is a nonconstant function whose coordinate functions reside in the Lipschitz closure of $\mathcal{F}$. 
(In particular, $h$ is $G$-invariant.)
Then for every $\epsilon>0$, there exists $f\colon X\to\mathbb{R}^n$ with coordinate functions in $\mathcal{F}$ such that $\operatorname{dist}(f^\downarrow)<\operatorname{dist}(h^\downarrow)+\epsilon$.
\end{lemma}

Let $\operatorname{Lip}_{\operatorname{h}}(\mathbb{R}^d)$ denote the vector space of positively homogeneous Lipschitz functions $\mathbb{R}^d\to\mathbb{R}$ with norm $\|\cdot\|_{\operatorname{Lip}}$.
(Here, $\|\cdot\|_{\operatorname{Lip}}$ is a norm and not merely a seminorm since there are no nonzero constant functions in $\operatorname{Lip}_{\operatorname{h}}(\mathbb{R}^d)$.)
This normed vector space is complete; indeed, it is a closed subspace of the space of Lipschitz functions $f\colon\mathbb{R}^d\to\mathbb{R}$ that satisfy $f(0)=0$, which in turn is a Banach space by Proposition~2.3(b) in~\cite{Weaver:18}.
As we will see, our main result follows from Lemma~\ref{lem.lipschitz closures suffice} and the following:
\begin{equation}
\label{eq.important containment}
h_1,\ldots,h_n\in
\overline{\operatorname{span}\mathcal{M}},
\end{equation}
where
\begin{itemize}
\item
$h_1,\ldots,h_n\colon\mathbb{R}^d\to\mathbb{R}$ are the coordinate functions of the optimal invariant $h$, defined in Example~\ref{ex.optimal coordinates},
\item
$\mathcal{M}$ is the set of max filter functions $\llangle[\cdot],[y]\rrangle\colon\mathbb{R}^d\to\mathbb{R}$ for $y\in\mathbb{R}^d$, and
\item
for any subset $\mathcal{F}\subseteq\operatorname{Lip}_{\operatorname{h}}(\mathbb{R}^d)$, $\overline{\mathcal{F}}$ denotes its topological closure in $\operatorname{Lip}_{\operatorname{h}}(\mathbb{R}^d)$.
\end{itemize}
We could not find a unified argument to explain~\eqref{eq.important containment} in all three cases of Example~\ref{ex.optimal coordinates}.
Interestingly, the cases appear to be fundamentally different, since the objects involved have case-dependent interactions with other function spaces of interest.
% https://www.youtube.com/watch?v=_-agl0pOQfs
To make this explicit, we need some additional notation:
\begin{itemize}
\item
$\operatorname{PL}_{\operatorname{h}}$ is the set of positively homogeneous piecewise linear functions $\mathbb{R}^d\to\mathbb{R}$,
\item
$\mathcal{P}$ is the set of polynomial functions $S^{d-1}\to\mathbb{R}$, where $S^{d-1}$ denotes the unit sphere in $\mathbb{R}^d$,
\item
given a $G$-invariant subset $S\subseteq\mathbb{R}^d$ and a set $\mathcal{F}$ of functions $S\to\mathbb{R}$, the set $\mathcal{F}^G$ consists of the $G$-invariant members of $\mathcal{F}$, and
\item
$\operatorname{ext}$ is the \textbf{homogeneous extension} map that extends functions $S^{d-1}\to\mathbb{R}$ to positively homogeneous functions $\mathbb{R}^d\to\mathbb{R}$, i.e., $\operatorname{ext}(f)(ru)=rf(u)$ for $r\geq0$ and $u\in S^{d-1}$.
\end{itemize}
In this notation, we will identify interactions between the following subspaces of $\operatorname{Lip}_{\operatorname{h}}(\mathbb{R}^d)^G$:
\[
\operatorname{span}\{h_i\}_{i=1}^n,
\qquad
\operatorname{ext}\mathcal{P}^G,
\qquad \operatorname{span}\mathcal{M},
\qquad \operatorname{PL}_{\operatorname{h}}^G.
\]

\begin{theorem}
\label{thm.main result ii}
Fix an input dimension $d$ and a group $G\leq\operatorname{O}(d)$.
\begin{itemize}
\item[(a)] 
If $d=2$ and $G$ is a nontrivial finite subgroup of $\operatorname{SO}(2)$, then
\[
\operatorname{span}\{h_i\}_{i=1}^n
\lneq \overline{\operatorname{ext}\mathcal{P}^G}
\lneq \overline{\operatorname{span}\mathcal{M}}
\qquad
\text{and}
\qquad
\operatorname{span}\mathcal{M}
=\operatorname{PL}_{\operatorname{h}}^G\text{.}
\]
\item[(b)]
If $d>2$ and $G=\{\pm I\}$, then 
\[
\operatorname{span}\{h_i\}_{i=1}^n
\lneq \overline{\operatorname{ext}\mathcal{P}^G}
\lneq \overline{\operatorname{span}\mathcal{M}}
\lneq \overline{\operatorname{PL}_{\operatorname{h}}^G}\text{.}
\]
\item[(c)]
If $G$ is a finite reflection group, then
\[
\operatorname{span}\{h_i\}_{i=1}^n
=\operatorname{span}\mathcal{M}
=\overline{\operatorname{span}\mathcal{M}}
\lneq \overline{\operatorname{PL}_{\operatorname{h}}^G}.
\]
Furthermore, if no nonzero vector is fixed by all of $G$ (i.e., G is essential), it holds that
\[
\overline{\operatorname{ext}\mathcal{P}^G}\cap\overline{\operatorname{span}\mathcal{M}}=\{0\}\text{.}
\]
\end{itemize}
\end{theorem}

Considering (a), (b), and (c) are so different from each other, a unified proof of~\eqref{eq.important containment} escapes us.
In particular, (b) and (c) show that it is not always the case that ``linear combinations of max filters approximate all $G$-invariant positively homogeneous piecewise linear functions'', so we cannot pass to this ``simpler'' function class to explain the phenomenon.
Also, (c) establishes that the containment is not always explained by ``linear combinations of max filters approximate all invariant polynomials on the sphere''.
In fact, since the max filters span a finite-dimensional subspace of functions in the case of (c), there is no norm under which the span of the max filters universally approximates $G$-invariant positively homogeneous Lipschitz functions.

We will prove Theorem~\ref{thm.main result ii} in Section~\ref{sec.proof of containments}.
For now, we show how it and Lemma~\ref{lem.lipschitz closures suffice} combine to imply Theorem~\ref{thm.main result}:

\begin{proof}[Proof of Theorem~\ref{thm.main result}]
Consider the appropriate $G$-invariant positively homogeneous map $h\colon\mathbb{R}^d\to\mathbb{R}^n$ given in Example~\ref{ex.optimal coordinates}.
Importantly, $\operatorname{dist}(h^\downarrow)=c_2(\mathbb{R}^d/G)$.
Let $\{h_i\}_{i=1}^n$ denote the coordinate functions of $h$.
By Theorem~\ref{thm.main result ii}, each coordinate function $h_i$ resides in the Lipschitz closure of $\mathcal{F}:=\operatorname{span}\mathcal{M}$.
By Lemma~\ref{lem.lipschitz closures suffice}, there exists $f\colon\mathbb{R}^d\to\mathbb{R}^n$ with coordinate functions in $\mathcal{F}$ such that 
\[
\operatorname{dist}(f^\downarrow)<\operatorname{dist}(h^\downarrow)+\epsilon=c_2(\mathbb{R}^d/G)+\epsilon.
\]
Finally, since each coordinate function of $f$ is a finite linear combination of max filters, we may express $f^\downarrow$ as the composition of a max filter bank $\Phi$ and a linear map $L$.

For (c), we may further take $f=h$, in which case $\epsilon=0$ and $m=d$.
Indeed, since $\operatorname{span}\{h_i\}_{i=1}^n=\operatorname{span}\mathcal{M}$ in Theorem~\ref{thm.main result ii}(c), 
it follows that each $h_i$ resides in the $d$-dimensional span of max filters $\mathbb{R}^d\to\mathbb{R}$.
As such, we may take the coordinate functions of $\Phi$ to be any basis of max filters, and then $L$ is determined by expressing each $h_i^\downarrow$ in terms of this basis.
\end{proof}

These results suggest some problems, which we leave for future research.
\begin{problem}
\label{problem.first problem}
~
\begin{itemize}
\item[(a)] 
Does Theorem~\ref{thm.main result} hold for every closed subgroup $G \leq \operatorname{O}(d)$?
\item[(b)]
Does every closed subgroup $G \leq \operatorname{O}(d)$ admit a finite target dimension $n$ for which there exists a $G$-invariant map $h \colon \mathbb{R}^d \to \mathbb{R}^n$ with $\operatorname{dist}(h^{\downarrow})=c_2(\mathbb{R}^d/G)$?
\item[(c)]
If the answer to (b) is ``yes'', can $h$ be taken to be positively homogeneous?
\item[(d)]
If the answers to the above are all ``yes'', is there a unified proof of~\eqref{eq.important containment}?
\end{itemize}
\end{problem}

\begin{problem}
For each closed subgroup $G\leq\operatorname{O}(d)$, characterize the functions that reside in the closure of the span of max filters.
\end{problem}

\subsection{Roadmap}

In the next section, we prove Lemma~\ref{lem.lipschitz closures suffice}.
In cases~(a) and~(b) of Example~\ref{ex.optimal coordinates}, it turns out that the coordinate functions of $h$ can be expressed as integral combinations of max filters, which in turn reside in the Lipschitz closure of linear combinations of max filters.
We establish this in Section~\ref{sec.integral}.
Section~\ref{sec.proof of containments} is dedicated to proving Theorem~\ref{thm.main result ii}, and we conclude with numerical results in Section~\ref{sec.numerical}.

\section{Continuity of distortion and proof of Lemma~\ref{lem.lipschitz closures suffice} }
\label{sec.proof of key lemma}

Before proving Lemma~\ref{lem.lipschitz closures suffice}, we first prove a couple of simpler lemmas that may be of independent interest; accordingly, we state these lemmas in an appropriate level of generality.

Our first lemma bears some resemblance to \textit{Weyl's inequality for singular values}, which states that
\[
|\sigma_i(A)-\sigma_i(B)|
\leq\sigma_1(A-B)
\qquad
\forall\,i,
\]
where $\sigma_1(M)\geq\sigma_2(M)\geq\cdots$ denote the singular values of a real matrix $M$.
Considering the top and bottom singular values of a tall, skinny matrix $M$ equal the optimal upper and lower Lipschitz bounds of $x\mapsto Mx$ (when viewed as a map between Euclidean spaces), the following lemma properly generalizes Weyl's inequality for the top and bottom singular values.

\begin{lemma}[Weyl's inequality for Lipschitz maps]
\label{lem.weyl}
Given a metric space $X$ with at least two points and a normed vector space $V$, then for any Lipschitz $f, g \colon X \to V$, it holds that
\[
\max\Big\{~
| \alpha(f) - \alpha(g) | ,~
| \beta(f) - \beta(g) | ~\Big\}
\leq \beta(f-g).
\]
\end{lemma}

\begin{proof}
Let $d$ denote the metric on $X$, and $\|\cdot\|$ the norm on $V$.
Given any $x \neq y$ in $X$, we write
\begin{align}
\frac{ \| f(x) - f(y) \|}{d(x,y)}
\nonumber
&= \frac{ \| g(x) - g(y) + f(x) - g(x) - f(y) + g(y) \| }{d(x,y)} \\[5 pt]
\label{eq.distance quotient}
&= \frac{ \| g(x) - g(y) + [f-g](x) - [f-g](y) \| }{d(x,y)}.
\end{align}
We obtain a lower bound on~\eqref{eq.distance quotient} from the reverse triangle inequality:
\[
\frac{ \| f(x) - f(y) \|}{d(x,y)}
\geq \frac{ \| g(x) - g(y) \|}{d(x,y)} - \frac{ \| [f-g](x) - [f-g](y) \| }{d(x,y)}
\geq \alpha(g) - \beta(f-g).
\]
It follows that $\alpha(f) \geq \alpha(g) - \beta(f-g)$.
By symmetry, we also have $\alpha(g) \geq \alpha(f) - \beta(f-g)$, and we may combine these to get
\[
|\alpha(f) - \alpha(g)| \leq \beta(f-g).
\]
Next, we obtain an upper bound on~\eqref{eq.distance quotient} from the triangle inequality:
\[
\frac{ \| f(x) - f(y) \|}{d(x,y)}
\leq \frac{ \| g(x) - g(y) \|}{d(x,y)} + \frac{ \| [f-g](x) - [f-g](y) \| }{d(x,y)}
\leq \beta(g) + \beta(f-g).
\]
It follows that $\beta(f) \leq \beta(g) + \beta(f-g)$.
By symmetry, we also have $\beta(g) \leq \beta(f) + \beta(f-g)$, and we may combine these to get
\[
|\beta(f) - \beta(g)| \leq \beta(f-g).
\qedhere
\]
\end{proof}

Next, we relate the Lipschitz norm of a $G$-invariant function to the Lipschitz norm of the corresponding function over the quotient space.

\begin{lemma}
\label{lem.down beta}
Given metric spaces $X$ and $Y$, consider any group $G$ that acts non-transitively on $X$ by isometries with topologically closed orbits.
For any $G$-invariant function $f\colon X\to Y$, $\beta(f^\downarrow)=\beta(f)$.
\end{lemma}

\begin{proof}
To establish $\beta(f)\leq\beta(f^\downarrow)$, it suffices to show
\[
\frac{d_Y(f(x),f(y))}{d_X(x,y)}
\leq\beta(f^\downarrow)
\qquad
\forall\,x,y\in X, ~x\neq y.
\]
This clearly holds if $[x]=[y]$, since then $f(x)=f(y)$.
Otherwise, we have $[x]\neq[y]$, in which case the inequality $d_{X/G}([x],[y])\leq d_X(x,y)$ implies
\[
\frac{d_Y(f(x),f(y))}{d_X(x,y)}
\leq\frac{d_Y(f(x),f(y))}{d_{X/G}([x],[y])}
\leq\beta(f^\downarrow).
\]
For the other inequality, take any $[x]\neq[y]$ in $X/G$.
Since the orbits of $G$ in $X$ are closed, there exists $g\in G$ such that $d_X(gx,y)=d_{X/G}([x],[y])\neq0$, and so
\[
\frac{d_Y(f^\downarrow([x]),f^\downarrow([y]))}{d_{X/G}([x],[y])}
=\frac{d_Y(f(gx),f(y))}{d_X(gx,y)}
\leq\beta(f).
\]
It follows that $\beta(f^\downarrow)\leq\beta(f)$.
\end{proof}

We are now ready to prove the main result of this subsection.

\begin{proof}[Proof of Lemma~\ref{lem.lipschitz closures suffice}]
Take any nonconstant function $h\colon X\to\mathbb{R}^n$ whose coordinate functions reside in the Lipschitz closure of $\mathcal{F}$, and fix $\epsilon>0$.
We may assume $\operatorname{dist}(h^\downarrow)$ is finite without loss of generality; in particular, $0<\alpha(h^\downarrow)\leq\beta(h^\downarrow)<\infty$.
For each $i\in[n]$, let $h_i\colon X\to\mathbb{R}$ denote the $i$th coordinate function of $h$, and select $f_i\in\mathcal{F}$ such that
\[
\beta(f_i-h_i)
<\frac{1}{\sqrt{n}}\cdot\min\left\{\frac{\epsilon\cdot\alpha(h^\downarrow)}{1+\operatorname{dist}(h^\downarrow)+\epsilon},\,\alpha(h^\downarrow)\right\}.
\]
Let $f\colon X\to\mathbb{R}^n$ denote the function whose $i$th coordinate function is $f_i$.
(Notably, $h$ and $f$ are both $G$-invariant.)
Then the triangle inequality and Lemmas~\ref{lem.weyl} and~\ref{lem.down beta} together give
\[
\beta(f^\downarrow)
\leq\beta(h^\downarrow)+|\beta(f^\downarrow)-\beta(h^\downarrow)|
\leq\beta(h^\downarrow)+\beta(f^\downarrow-h^\downarrow)
=\beta(h^\downarrow)+\beta(f-h),
\]
and likewise,
\[
\alpha(f^\downarrow)
\geq\alpha(h^\downarrow)-|\alpha(f^\downarrow)-\alpha(h^\downarrow)|
\geq\alpha(h^\downarrow)-\beta(f^\downarrow-h^\downarrow)
=\alpha(h^\downarrow)-\beta(f-h).
\]
To estimate $\beta(f-h)$, note that for any $x\neq y$ in $X$, we have
\[
\frac{\|(f-h)(x)-(f-h)(y)\|}{d(x,y)}
\leq\frac{\sqrt{n}\cdot\max_i|(f_i-h_i)(x)-(f_i-h_i)(y)|}{d(x,y)}
\leq\sqrt{n}\cdot\max_i\beta(f_i-h_i),
\]
where $d$ denotes the metric on $X$.
Thus,
\[
\beta(f-h)
\leq\sqrt{n}\cdot\max_i\beta(f_i-h_i)
<\min\left\{\frac{\epsilon\cdot\alpha(h^\downarrow)}{1+\operatorname{dist}(h^\downarrow)+\epsilon},\,\alpha(h^\downarrow)\right\}.
\]
In particular, the second argument of the above minimum implies $0<\alpha(f^\downarrow)\leq\beta(f^\downarrow)<\infty$ and $\alpha(h^\downarrow) - \beta(f-h) > 0$, while the first argument implies
\[
\operatorname{dist}(f^\downarrow)
=\frac{\beta(f^\downarrow)}{\alpha(f^\downarrow)}
\leq\frac{\beta(h^\downarrow)+\beta(f-h)}{\alpha(h^\downarrow)-\beta(f-h)}
<\operatorname{dist}(h^\downarrow)+\epsilon.
\qedhere
\]
\end{proof}

\section{Integral combinations of max filters}
\label{sec.integral}

In this section, we focus on cases (a) and (b) of Example~\ref{ex.optimal coordinates}.
For these cases, we prove \eqref{eq.important containment}, i.e., the optimal coordinate functions $h_1,\ldots,h_n$ given in Example~\ref{ex.optimal coordinates} reside in $\overline{\operatorname{span}\mathcal{M}}$, namely, the Lipschitz closure of the span of max filters.
In the next subsection (specifically, in Theorem~\ref{thm.integral combos are in lipschitz closure of linear of max filter}), we show that it suffices to express each $h_i$ as an \textit{integral} combination of max filters.
(Here, we mean \textit{integral} as in $\int$, not $\mathbb{Z}$.)
We then apply this idea in Subsections~\ref{subsec:planar} and~\ref{subsec:real phase retrieval} by expressing each $h_i$ (and in fact, every invariant polynomial on the sphere) as an integral combination of max filters in cases (a) and (b) of Example~\ref{ex.optimal coordinates}, respectively.
We accomplish this by leveraging tools from harmonic analysis.
Throughout, $\omega$ denotes a Haar measure on $S^{d-1}$.

\subsection{Primary technology}

In this subsection, we introduce the primary technology that we will use in the remainder of this section.
First, we enunciate a useful generalization of the fact that the Lipschitz norm of a smooth function $f\colon\mathbb{R}^d\to\mathbb{R}$ is the supremum of the $2$-norm of the gradient of $f$.
In particular, this allows us to use familiar tools from calculus to estimate Lipschitz norms.

\begin{proposition}[strong Rademacher's theorem, see Theorem~1.41 in~\cite{Weaver:18}]
\label{prop.rademacher}
Let $f\colon\mathbb{R}^d\to\mathbb{R}$ be Lipschitz.
Then $f$ is differentiable almost everywhere, and
\[
\|f\|_{\operatorname{Lip}}
=\esssup_{x\in\mathbb{R}^d}\|\nabla f(x)\|.
\]
\end{proposition}

In the special case where $f\colon\mathbb{R}^d\to\mathbb{R}$ is a max filter, Proposition~\ref{prop.rademacher} requires the gradient of $f$, which is supplied by the following:

\begin{proposition}[gradient of max filter, see Theorem~29 in~\cite{CahillIM:24}]
\label{prop.grad max filter}
Given a finite group $G\leq\operatorname{O}(d)$ and $y\in\mathbb{R}^d$, it holds that the max filter $\llangle [\cdot],[y]\rrangle\colon\mathbb{R}^d\to\mathbb{R}$ is differentiable at $x\in\mathbb{R}^d$ precisely when there is a unique $z\in[y]$ for which $\langle x,z\rangle=\llangle [x],[y]\rrangle$, in which case 
\[
\nabla\llangle [\cdot],[y]\rrangle(x)
=z
=\argmax_{u\in [y]}\,\langle x,u\rangle.
\]
\end{proposition}

Next, to determine the gradient of an integral combination of max filters, it turns out that we can \textit{differentiate under the integral sign}, meaning the desired gradient is just the corresponding integral combination of max filter gradients:

\begin{lemma}
\label{lem.differentiation under the integral sign}
Given a finite group $G\leq\operatorname{O}(d)$ and $q\in L^1(S^{d-1})$, consider $p\colon\mathbb{R}^d\to\mathbb{R}$ defined by
\[
p(x)
=\int_{S^{d-1}}q(y)\,\llangle[x],[y]\rrangle\,d\omega(y).
\]
At every $x\in\mathbb{R}^d$ with trivial stabilizer in $G$, it holds that $p$ is differentiable with gradient
\[
\nabla p(x)
=\int_{S^{d-1}}q(y)\,\Big(\argmax_{u\in [y]}\,\langle x,u\rangle\Big)\,d\omega(y).
\]
In particular, the $\argmax$ above is unique for almost every $y\in S^{d-1}$.
\end{lemma}

\begin{proof}
First, we prove the ``in particular'' part.
Consider the set $B$ of all $y \in S^{d-1}$ for which $\argmax_{u \in [y]}\,\langle x,u\rangle$ is not unique.
For any $y \in B$, there exist $g,h \in G$ with $g \neq h$ for which $\langle x, gy \rangle = \langle x, hy \rangle$, that is, $\langle g^{-1}x, y \rangle = \langle h^{-1}x, y \rangle$.
Since $x$ is trivially stabilized in $G$, it follows that $y$ resides in the boundary of a Voronoi cell of $[x] \subseteq \mathbb{R}^d$.
Thus, $B$ is contained in the intersection of $S^{d-1}$ with the (finite) union of such boundaries, which has measure zero in $S^{d-1}$.

Next, fix a standard basis element $e_i$ and a sequence $\{s_n\}_{n=1}^\infty$ of nonzero real numbers such that $s_n\to0$.
Given any $y \in S^{d-1}\setminus B$, denote $u(y) \in [y]$ for $\argmax_{u\in [y]}\,\langle x,u\rangle$, which is unique since $y \notin B$.
Then, considering the $i$th coordinate of $\nabla p(x)$, it suffices to show that
\[
\lim_{n\to\infty}\frac{p(x+s_ne_i)-p(x)}{s_n}
=\int_{S^{d-1}}q(y)\,\langle e_i,u(y)\rangle\,d\omega(y),
\]
or equivalently, after re-expressing the left-hand side,
\[
\lim_{n\to\infty}\int_{S^{d-1}}q(y)\,\frac{\llangle [x+s_ne_i],[y]\rrangle - \llangle [x],[y]\rrangle}{s_n}\,d\omega(y)
= \int_{S^{d-1}}q(y)\,\langle e_i,u(y)\rangle\,d\omega(y).
\]
We will prove the above identity with an application of the Lebesgue dominated convergence theorem.
Indeed, it suffices to verify that the functions
\[
f_n(y)
:=q(y)\,\frac{\llangle [x+s_ne_i],[y] \rrangle - \llangle [x],[y] \rrangle}{s_n}
\]
have the properties that
\begin{itemize}
\item[(i)]
for every $y \in S^{d-1}$, $|f_n(y)| \leq |q(y)|$, and
\item[(ii)]
for every $y\in S^{d-1}\setminus B$, $\displaystyle\lim_{n\to \infty} f_n(y) = q(y)\,\langle e_i,u(y)\rangle$.
\end{itemize}
For (i), note that $x \mapsto \llangle [x],[y]\rrangle$ is $1$-Lipschitz for every $y\in S^{d-1}$, and so
\[
|f_n(y)|
= |q(y)|\cdot\frac{|\llangle [x+s_ne_i],[y] \rrangle - \llangle [x],[y] \rrangle|}{\|(x+s_n e_i) - x\|}
\leq |q(y)|.
\]
For (ii), consider any $y \in S^{d-1} \setminus B$, and recall that $u(y)$ uniquely maximizes the linear functional $\langle x,\cdot\rangle$ over the orbit $[y]$.
Among the Voronoi cells for the orbit $[y]$, $x$ resides in the open Voronoi cell $V$ containing $u(y)$.
For $n$ sufficiently large, it holds that $x+s_n e_i$ also resides in $V$, in which case
\[
f_n(y)
=q(y)\,\frac{\llangle [x+s_ne_i],[y] \rrangle - \llangle [x],[y] \rrangle}{s_n}
=q(y)\,\frac{\langle x+s_ne_i,u(y)\rangle-\langle x,u(y)\rangle}{s_n}
=q(y)\,\langle e_i,u(y)\rangle.
\qedhere
\]
\end{proof}

Finally, we establish the main result of this subsection, namely, that (sufficiently nice) integral combinations of max filters reside in the Lipschitz closure of the span of max filters:

\begin{theorem}
\label{thm.integral combos are in lipschitz closure of linear of max filter}
Suppose $G\leq\operatorname{O}(d)$ is finite, and given a $G$-invariant positively homogeneous function $p\colon \mathbb{R}^d\to\mathbb{R}$, suppose there exists a Lipschitz function $q\colon S^{d-1}\to\mathbb{R}$ such that
\begin{equation}
\label{eq.p as an integral combination of max filters}
p(x)
=\int_{S^{d-1}}q(y)\,\llangle[x],[y]\rrangle\,d\omega(y)
\end{equation}
for all $x\in S^{d-1}$ (hence all $x \in \mathbb{R}^d$, by positive homogeneity).
Then $p$ resides in the Lipschitz closure of the span of max filters $x \mapsto \llangle [x],[y] \rrangle$ for $y \in S^{d-1}$, i.e., $p\in\overline{\operatorname{span}\mathcal{M}}$.
\end{theorem}

\begin{proof}
Let $\epsilon>0$ be arbitrary, and consider a partition $S^{d-1}=\bigsqcup_{k=1}^n I_k$ such that each $I_k$ is measurable and connected with chordal diameter at most $\epsilon$.
For each $k$, take a sample $y_k\in I_k$.
For each $x\in \mathbb{R}^d$, we approximate the integral combination $p(x)$ with the Riemann sum
\[
\hat{p}(x)
:=\sum_{k=1}^n q(y_k)\,\llangle[x],[y_k]\rrangle\,|I_k|,
\]
where we abbreviate $|\cdot|:=\omega(\cdot)$.
In particular, $\hat{p}\colon\mathbb{R}^d\to\mathbb{R}$ resides in $\operatorname{span}\mathcal{M}$.
It suffices to show $\|\hat{p}-p\|_{\operatorname{Lip}}=O(\epsilon)$.

Note that $p$ is Lipschitz since $q$ is continuous (and hence absolutely integrable) and $x\mapsto\llangle[x],[y]\rrangle$ is $1$-Lipschitz for every unit vector $y$:
\[
|p(x)-p(x')|
\leq \int_{S^{d-1}}|q(y)|\,|\llangle[x],[y]\rrangle-\llangle[x'],[y]\rrangle|\,d\omega(y)
\leq\|q\|_{L^1(S^{d-1})}\cdot\|x-x'\|.
\]
Meanwhile, $\hat{p}$ is Lipschitz on account of being a finite linear combination of (Lipschitz) max filters $x\mapsto \llangle[x],[y_k]\rrangle$.
Then Proposition~\ref{prop.rademacher} gives that $p$ and $\hat{p}$ are differentiable almost everywhere, and furthermore,
\begin{equation}
\label{eq.rademacher}
\|\hat{p}-p\|_{\operatorname{Lip}}
=\esssup_{x\in\mathbb{R}^d}\|\nabla \hat{p}(x)-\nabla p(x)\|.
\end{equation}

To provide a set on which to estimate the $\esssup$ above, consider the subset $S \subseteq \mathbb{R}^d$ consisting of all points at which $\llangle [\cdot],[y_k]\rrangle$ is differentiable for every $k \in [n]$, and that are furthermore trivially stabilized in $G$. 
We claim $S$ contains almost every point in $\mathbb{R}^d$.
Indeed, $\llangle [\cdot],[y_k]\rrangle$ is nondifferentiable at $z\in\mathbb{R}^d$ only if there exist $u\neq u'$ in $[y_k]$ for which $\langle z,u\rangle=\langle z,u'\rangle$ (see Proposition~\ref{prop.grad max filter}), in which case $z$ resides in the orthogonal complement of $u-u'$.
Thus, the points of non-differentiability for $\llangle [\cdot],[y_k]\rrangle$ lie in the union of finitely many hyperplanes, which has measure zero.
Also, the points that are stabilized by a non-identity element $g\in G$ reside in the proper subspace $\ker(g- I)$, which has measure zero.
Since $G$ is finite by assumption, it follows that almost every point in $\mathbb{R}^d$ is trivially stabilized in $G$.
This proves the claim.

Now select any $x \in S$.
Then denoting 
\[
v_x(y)
:=q(y)\argmax_{u\in [y]}\langle x,u\rangle,
\]
Lemma~\ref{lem.differentiation under the integral sign} above and Proposition~\ref{prop.grad max filter} respectively give that $p$ and $\hat{p}$ are both differentiable at $x$, with
\[
\nabla p(x)
=\int_{S^{d-1}}v_x(y)\,d\omega(y),
\qquad
\nabla \hat{p}(x)
:=\sum_{k=1}^n v_x(y_k)\,|I_k|.
\]
In particular,
\begin{equation}
\label{eq.bound of norm of diff of grads}
\|\nabla\hat{p}(x)-\nabla p(x)\|
=\bigg\|\sum_{k=1}^n \int_{I_k} \big( v_x(y)-v_x(y_k)\big)\,d\omega(y)\bigg\|
\leq\sum_{k=1}^n \int_{I_k} \big\| v_x(y)-v_x(y_k)\big\|\,d\omega(y).
\end{equation}

We will split the terms of \eqref{eq.bound of norm of diff of grads} into two types and bound them with different strategies.
Specifically, let $A\subseteq[n]$ denote the set of indices $k\in[n]$ for which $I_k$ resides in an open Voronoi cell in $\mathbb{R}^d$ relative to the orbit $[x]$, and let $B:=[n]\setminus A$ denote the remaining indices.
As we will see, the $\sum_{k\in A}$ portion of \eqref{eq.bound of norm of diff of grads} is small since $q$ is Lipschitz and each $I_k$ has small diameter, while the $\sum_{k\in B}$ portion is small since $q$ is bounded and $\bigcup_{k\in B}I_k$ has small measure.

Suppose $k\in A$, and let $V$ denote the open Voronoi cell in $\mathbb{R}^d$ relative to the orbit $[x]$ that contains $I_k$.
There is a group element $g\in G$ such that for every $y\in I_k\subseteq V$, it holds that $gy$ resides in the Voronoi cell containing $x$.
As such, $v_x(y)=q(y)gy$, and so
\begin{align*}
\|v_x(y)-v_x(y_k)\|
&=\|q(y)gy-q(y_k)gy_k\|\\
&\leq\|q(y)gy-q(y_k)gy\|+\|q(y_k)gy-q(y_k)gy_k\|\\
&\leq |q(y)-q(y_k)|+|q(y_k)|\|y-y_k\|\\
&\leq \|q\|_{\operatorname{Lip}}\cdot\epsilon+\|q\|_{\infty}\cdot\epsilon.
\end{align*}
It follows that
\[
\sum_{k\in A} \int_{I_k} \big\| v_x(y)-v_x(y_k)\big\|\,d\omega(y)
\leq\sum_{k\in A} |I_k|\cdot\big(\|q\|_{\operatorname{Lip}}+\|q\|_{\infty}\big)\cdot\epsilon
\leq\omega_{d-1}\cdot \big(\|q\|_{\operatorname{Lip}}+\|q\|_{\infty}\big)\cdot\epsilon,
\]
where $\omega_{d-1}$ denotes the surface area of $S^{d-1}$.

For the $k\in B$ case, note that
\[
\sum_{k\in B} \int_{I_k} \big\| v_x(y)-v_x(y_k)\big\|\,d\omega(y)
\leq \bigg|\bigcup_{k\in B}I_k\bigg|\cdot\sup_{\substack{k\in [n]\\y\in S^{d-1}}}\|v_x(y)-v_x(y_k)\|
\leq \bigg|\bigcup_{k\in B}I_k\bigg|\cdot2\|q\|_\infty,
\]
where the last step follows from the triangle inequality:
for $k\in [n]$ and $y\in S^{d-1}$, it holds that
\[
\|v_x(y)-v_x(y_k)\|
\leq\|v_x(y)\|+\|v_x(y_k)\|
\leq2\|q\|_\infty.
\]
It remains to bound the measure of $\bigcup_{k\in B}I_k$, which we accomplish by bounding the measure of a convenient superset.

To this end, consider any $k\in B$.
Since $I_k$ is connected by assumption, it nontrivially intersects one of the linear hyperplanes $\{H_j\}_{j=1}^m$ that together contain the facets of the Voronoi cells relative to $[x]$.
(Notably, $m\leq\binom{|G|}{2}$, regardless of $x$.)
In particular, since $I_k$ has diameter at most $\epsilon$ by assumption, it resides in an $\epsilon$-thickened version of some $H_j$.
Meanwhile, the spherical measure of this thickened version of $H_j$ can be obtained by first integrating out the ``lines of latitude'' that are parallel to $H_j$ by way of a coarea formula:
\[
\big|(H_j+B_0(\epsilon))\cap S^{d-1}\big|
=\omega_{d-2}\int_{-\epsilon}^\epsilon (1-t^2)^{\frac{d-3}{2}}\,dt
\leq 2\epsilon\omega_{d-2},
\]
where $\omega_{d-2}>0$ quantifies the contribution of the equator $S^{d-2} \cong H_j\cap S^{d-1}\subseteq S^{d-1}$; in particular, it does not depend on $j$. 
It follows that
\[
\bigg|\bigcup_{k\in B}I_k\bigg|
\leq\sum_{j=1}^m\big|(H_j+B_0(\epsilon))\cap S^{d-1}\big|
\leq\binom{|G|}{2}\cdot2\epsilon\omega_{d-2}.
\]

Putting everything together, we have
\begin{align*}
\|\nabla\hat{p}(x)-\nabla p(x)\|
&\leq\sum_{k\in A} \int_{I_k} \big\| v_x(y)-v_x(y_k)\big\|\,d\omega(y)+\sum_{k\in B} \int_{I_k} \big\| v_x(y)-v_x(y_k)\big\|\,d\omega(y)\\
&\leq \Big(|S^{d-1}|\cdot\big(\|q\|_{\operatorname{Lip}}+\|q\|_{\infty}\big)+\tbinom{|G|}{2}\cdot 2\omega_{d-2}\cdot 2\|q\|_\infty\Big)\cdot\epsilon.
\end{align*}
Since $x\in S$ was arbitrary, it follows from \eqref{eq.rademacher} that $\|\hat{p}-p\|_{\operatorname{Lip}}=O(\epsilon)$, as desired.
\end{proof}

\subsection{Application 1: Planar rotations}
\label{subsec:planar}

In this subsection, we establish~\eqref{eq.important containment} in the special case where $h_1,\ldots,h_n$ are given by Example~\ref{ex.optimal coordinates}(a).
By Theorem~\ref{thm.integral combos are in lipschitz closure of linear of max filter}, it suffices to express each $h_i$ as a suitably nice integral combination of max filters.
This follows from the following (after accounting for some subtleties, which we do later) since each $h_i$ is a $G$-invariant polynomial when restricted to $S^1$:

\begin{theorem}
\label{thm.integral combo for rotations of the plane}
Take a nontrivial finite subgroup $G\leq \operatorname{SO}(2)$.
For every $G$-invariant polynomial $p$, there is a $G$-invariant polynomial $q$ such that
\[
p(x)
=\int_{S^1} q(y) \, \llangle [x],[y] \rrangle \, d\omega(y)
\]
for all $x\in S^1$.
\end{theorem}

\begin{proof}
Without loss of generality, we take $\omega$ to be the uniform probability measure on $S^1$.
Writing $r:=|G|>1$ and $x=(\cos\theta,\sin\theta)$, then $\llangle [x],[e_1] \rrangle$ is represented by the function $f\colon\mathbb{R}\to\mathbb{R}$ defined by $f(\theta)=\cos\theta$ for $\theta\in[-\frac{\pi}{r},\frac{\pi}{r}]$ and extended $\frac{2\pi}{r}$-periodically.
Notably, if we put $y=(\cos\varphi,\sin\varphi)$, then $\llangle [x],[y]\rrangle=f(\theta-\varphi)$.
Next, $g(\theta)=p(\cos\theta,\sin\theta)$ defines a $\frac{2\pi}{r}$-periodic trigonometric polynomial $g\colon\mathbb{R}\to\mathbb{R}$, and we seek a $\frac{2\pi}{r}$-periodic trigonometric polynomial $c\colon\mathbb{R}\to\mathbb{R}$ such that
\[
g(\theta)
=\frac{1}{2\pi}\int_0^{2\pi} c(\varphi)\,f(\theta-\varphi)\,d\varphi
\qquad
\forall\,\theta\in\mathbb{R}.
\]
Our task simplifies in the Fourier domain:
\[
\hat{g}(k)
=\hat{c}(k)\cdot\hat{f}(k)
\qquad
\forall\,k\in\mathbb{Z}.
\]
This suggests the following choice for $c$:
\begin{equation}
\label{eq.fourier series of coeff function}
\hat{c}(k)
=\left\{\begin{array}{cl}
\displaystyle\frac{\hat{g}(k)}{\hat{f}(k)}&\text{if }\hat{f}(k)\neq 0\\[12 pt]
0&\text{else.}
\end{array}\right.
\end{equation}
This choice works since $\hat{f}(k)\neq 0$ if and only if $r$ divides $k$.
Indeed, the ``only if'' direction follows from the fact that $f$ is $\frac{2\pi}{r}$-periodic, while the ``if'' direction follows from computing the Fourier series (with a bit of grueling integration and the aid of a product-to-sum formula):
\[
\hat{f}(rm)
=(-1)^{m+1} \cdot \frac{\frac{r}{\pi}\sin(\frac{\pi}{r})}{(rm)^2-1}
\qquad
\forall\,m\in\mathbb{Z}.
\qedhere
\]
\end{proof}

To illustrate Theorem~\ref{thm.integral combo for rotations of the plane}, we apply it to express the optimal coordinate functions $h_1$, $h_2$, and $h_3$ in the simplest case, where $G$ has order $2$.
(Historically, this example marked the beginning of this paper; with the help of a computer experiment, we managed to approximate the optimal coordinate functions as linear combinations of max filters, and the mysterious integral identities that fell out compelled us to investigate further.)

\begin{example}
In the special case where $|G|=2$, i.e., $G=\{\pm I\}$, a minimum-distortion Euclidean embedding of $\mathbb{R}^2/G$ is given by the homogeneous extension of
\[
x
\mapsto xx^\top
=
\left[\begin{array}{cc}
x_1^2 & x_1x_2 \\
x_1x_2 & x_2^2
\end{array}\right]
=\left[\begin{array}{cc}
\cos^2\theta & \cos\theta\sin\theta \\
\cos\theta\sin\theta & \sin^2\theta
\end{array}\right]
\]
for $x=\left[\begin{smallmatrix}\cos\theta\\\sin\theta\end{smallmatrix}\right]$ with $\theta\in[0,2\pi)$.
(See Example~\ref{ex.optimal coordinates}(b), and note that the embedding in this $d=2$ case can be viewed as isometrically equivalent to the embedding in Example~\ref{ex.optimal coordinates}(a).)
We will express these coordinate functions as integral combinations of max filters.
In this case, the max filter of $\left[\begin{smallmatrix}\cos\theta\\\sin\theta\end{smallmatrix}\right]$ against the template $\left[\begin{smallmatrix}\cos\varphi\\\sin\varphi\end{smallmatrix}\right]$ is given by
\[
f(\theta-\varphi)
:=|\cos(\theta-\varphi)|.
\]
This matches the notation in the proof of Theorem~\ref{thm.integral combo for rotations of the plane} above since $f\colon\mathbb{R}\to\mathbb{R}$ is $\pi$-periodic with $f(\theta)=\cos\theta$ for $\theta\in[-\frac{\pi}{2},\frac{\pi}{2}]$.
We follow the recipe~\eqref{eq.fourier series of coeff function} to derive a coefficient function (namely, $c$ in the proof) for each coordinate function ($g$ in the proof).
First, we appeal to standard trigonometric identities to obtain the Fourier series
\begin{align*}
\cos^2\theta
&=\tfrac{1}{2}+\tfrac{1}{4}e^{i2\theta}+\tfrac{1}{4}e^{-i2\theta}, \\
\sin^2\theta
&=\tfrac{1}{2}-\tfrac{1}{4}e^{i2\theta}-\tfrac{1}{4}e^{-i2\theta}, \\
\cos\theta\sin\theta
&=\phantom{\tfrac{1}{2}}-\tfrac{i}{4}e^{i2\theta}+\tfrac{i}{4}e^{-i2\theta}.
\end{align*}
Meanwhile, the last display of the proof of Theorem~\ref{thm.integral combo for rotations of the plane} gives the Fourier series of $f$:
\[
f(\theta)
=\frac{2}{\pi}\sum_{m=-\infty}^\infty\frac{(-1)^{m}}{1-4m^2}\cdot e^{i2m\theta}.
\]
Letting $c$ denote the coefficient function for $\cos^2\theta$, we apply equation~\eqref{eq.fourier series of coeff function} to get the Fourier series of $c$:
\[
c(\theta)
= \sum_{k=-\infty}^\infty \hat{c}(k)\cdot e^{ik\theta}
= \tfrac{\pi}{4}+\tfrac{3\pi}{8}e^{i2\theta}+\tfrac{3\pi}{8}e^{-i2\theta}
= \tfrac{\pi}{4}+\tfrac{3\pi}{4}\cos2\theta,
\]
which gives the (nontrivial!)\ identity
\[
\cos^2\theta
=\frac{1}{2\pi}\int_0^{2\pi}c(\varphi)f(\theta-\varphi)d\varphi
=\int_0^{\pi}\Big(\tfrac{1}{4}+\tfrac{3}{4}\cos2\varphi\Big)~|\cos(\theta-\varphi)|~d\varphi.
\]
Similar calculations give expressions for the other coordinate functions:
\begin{align*}
\sin^2\theta
&=\int_0^{\pi}\Big(\tfrac{1}{4}-\tfrac{3}{4}\cos2\varphi\Big)~|\cos(\theta-\varphi)|~d\varphi,\\
\cos\theta\sin\theta
&=\int_0^{\pi}\tfrac{3}{4}\sin2\varphi~|\cos(\theta-\varphi)|~d\varphi.
\end{align*}
\end{example}

\subsection{Application 2: Real phase retrieval}
\label{subsec:real phase retrieval}

Next, we establish~\eqref{eq.important containment} in the special case where $h_1,\ldots,h_n$ are given by Example~\ref{ex.optimal coordinates}(b).
Again, by Theorem~\ref{thm.integral combos are in lipschitz closure of linear of max filter}, it suffices to express each $h_i$ as a suitably nice integral combination of max filters.
Since $G=\{\pm I\}$, the max filter $\llangle [x],[y]\rrangle$ is simply $|\langle x,y\rangle|$.
As before, each $h_i$ is a $G$-invariant (i.e., even) polynomial when restricted to the sphere.
What follows is the main result of this subsection:

\begin{theorem}
\label{thm.integral of max filters for pmI invariant polynomials}
Take $d>2$ and $G=\{\pm I\}\leq \operatorname{O}(d)$.
For every $G$-invariant (i.e., even) polynomial $p$, there is a $G$-invariant polynomial $q$ such that
\[
p(x)
=\int_{S^{d-1}} q(y) \, |\langle x,y\rangle| \, d\omega(y)
\]
for all $x\in S^{d-1}$.
\end{theorem}

Our proof of Theorem~\ref{thm.integral of max filters for pmI invariant polynomials} follows the same idea as our proof of Theorem~\ref{thm.integral combo for rotations of the plane}, but the correct notion of Fourier series in this case is given by the theory of \textit{spherical harmonics};
we recommend Chapter~IV, Section~2 of Stein and Weiss~\cite{SteinW:71} as a reference.
Throughout, we assume $d>2$ in order to avoid the qualitatively different sphere $S^1$, which in turn is covered by Theorem~\ref{thm.integral combo for rotations of the plane}.
Recall the decomposition
\[
L^2(S^{d-1})
=\bigoplus_{k=0}^\infty\mathcal{H}_k,
\]
where $\mathcal{H}_k$ denotes the subspace of homogeneous harmonic polynomial functions of degree $k$ over $S^{d-1}$.
Since $\mathcal{H}_k$ is finite dimensional, for each $x\in S^{d-1}$, the evaluation map $\mathcal{H}_k\to\mathbb{C}$ defined by $p\mapsto p(x)$ can be represented as an inner product against some $Z_x^{(k)}\in\mathcal{H}_k$:
\begin{equation}
\label{eq.reproducing property}
p(x)
=\langle p,Z_x^{(k)}\rangle_{L^2(S^{d-1})}
\end{equation}
This representer function $Z_x^{(k)}$ is known as the \textit{zonal spherical harmonic of degree $k$ with pole $x$}.
It turns out that $Z_x^{(k)}$ is real valued, and furthermore, $Z_{gx}^{(k)}(gy)=Z_{x}^{(k)}(y)$ for every $g\in \operatorname{SO}(d)$.
Up to scaling, $x\mapsto Z_x^{(k)}$ is the unique function $S^{d-1}\to\mathcal{H}_k$ that exhibits this invariance to $\operatorname{SO}(d)$:

\begin{proposition}[Corollary~2.13 of {\cite[Chapter~IV]{SteinW:71}}]
\label{prop.zonal symmetry}
Fix $d>2$ and $k \geq 0$, and consider any $F\colon S^{d-1}\to\mathcal{H}_k$ such that
\[
F(gx)(gy)=F(x)(y)
\qquad
\forall\, x,y\in S^{d-1},~g\in \operatorname{SO}(d).
\]
Then there exists $c\in\mathbb{C}$ such that $F(x)=cZ_x^{(k)}$ for every $x\in S^{d-1}$.
\end{proposition}

\begin{proof}[Proof of Theorem~\ref{thm.integral of max filters for pmI invariant polynomials}]
First, note that for each nonnegative integer $k$, $\mathcal{H}_k$ is an $\operatorname{SO}(d)$-invariant subspace of $L^2(S^{d-1})$, and so the orthogonal projection map $\pi_k\colon L^2(S^{d-1})\to\mathcal{H}_k$ is $\operatorname{SO}(d)$-equivariant.
As such, for each $x\in S^{d-1}$, the function $F_k(x):=\pi_k\big(|\langle x,\cdot\rangle|\big) \in \mathcal{H}_k$ satisfies
\[
F_k(gx)(gy)
=\pi_k\big(|\langle gx,\cdot\rangle|\big)(gy)
=\big(g^{-1}\cdot\pi_k\big(|\langle gx,\cdot\rangle|\big)\big)(y)
=\pi_k\big(|\langle gx,g\,\cdot\rangle|\big)(y)
=\pi_k\big(|\langle x,\cdot\rangle|\big)(y)
=F_k(x)(y)
\]
for every $y\in S^{d-1}$ and $g\in \operatorname{SO}(d)$.
By Proposition~\ref{prop.zonal symmetry}, there exists a sequence $\{c_k\}_{k=0}^\infty$ such that, for every $x \in S^{d-1}$,
\[
|\langle x,\cdot\rangle|
=\sum_{k=0}^\infty\pi_k\big(|\langle x,\cdot\rangle|\big)
=\sum_{k=0}^\infty F_k(x)
=\sum_{k=0}^\infty c_kZ_x^{(k)},
\]
with convergence in $L^2(S^{d-1})$.
Furthermore, each $c_k$ is real since the function $|\langle x,\cdot\rangle|$ and the mutually orthogonal functions $Z_x^{(k)}$ are all real valued.

Next, given an even polynomial function $p$, we may decompose $p=\sum_k p_k$, where each $p_k$ resides in $\mathcal{H}_k$.
This sum is finite since it is an equality of polynomial functions.
Furthermore, since $p$ is even, it follows that $p_k$ is zero whenever $k$ is odd.
As such, we may assume $p\in\mathcal{H}_{k}$ with $k=2m$ without loss of generality.
Then for every $x \in S^{d-1}$, it holds that
\[
\int_{S^{d-1}} p(y)\,|\langle x,y\rangle|\,d\omega(y)
=\big\langle p,|\langle x,\cdot\rangle|\big\rangle_{L^2(S^{d-1})}
=\langle p,c_kZ_x^{(k)}\rangle_{L^2(S^{d-1})}
=c_k\langle p,Z_x^{(k)}\rangle_{L^2(S^{d-1})}
=c_kp(x),
\]
where the last step uses the reproducing property~\eqref{eq.reproducing property}.
As such, we may take $q:=p/c_k$, provided we show that $c_k\neq0$.
This is the focus of the remainder of the proof.

Fix some $x \in S^{d-1}$.
Since $c_kZ_x^{(k)}$ is the projection of $|\langle x,\cdot\rangle|$ onto $\mathcal{H}_k$, it is also the projection of $|\langle x,\cdot\rangle|$ onto the span of $Z_x^{(k)}$, and so
\[
c_k
=\frac{\big\langle |\langle x,\cdot\rangle|,Z_x^{(k)}\big\rangle_{L^2(S^{d-1})}}{\|Z_x^{(k)}\|_{L^2(S^{d-1})}^2}
=\Big\langle |\langle x,\cdot\rangle|,\tfrac{1}{Z_x^{(k)}(x)}Z_x^{(k)}\Big\rangle_{L^2(S^{d-1})},
\]
where the last step follows from the reproducing property \eqref{eq.reproducing property}, namely, $\|Z_x^{(k)}\|_{L^2(S^{d-1})}^2=Z_x^{(k)}(x)$.
Since $|\langle x,\cdot\rangle|$ and $\tfrac{1}{Z_x^{(k)}(x)}Z_x^{(k)}$ are both constant along the ``lines of latitude'' orthogonal to the pole $x$, an application of the \textit{smooth coarea formula} reduces the desired inner product to an integral over the $x$-component $t\in[-1,1]$ of $y\in S^{d-1}$:
\[
c_k
=\omega_{d-2}\int_{-1}^1 |t|\, G_k(t)\, (1-t^2)^{\frac{d-3}{2}}\,dt,
\]
where $\omega_{d-2}>0$ denotes the appropriate measure of contribution by the equator $S^{d-2}$, and $G_k$ denotes the \textit{Gegenbauer polynomial} of degree~$k$; see Theorem~2.14 in~\cite[Chapter~IV]{SteinW:71}.
The sequence $\{G_j\}_{j=0}^\infty$ is uniquely determined by iteratively taking $G_j$ to be any real-valued polynomial of degree $j$ that is orthogonal to $\{G_i\}_{i<j}$ over the measure $(1-t^2)^{\frac{d-3}{2}}dt$ on $[-1,1]$, and then normalizing so that $G_j(1)=1$.

We conclude by using the fact that $k=2m$ is even.
Since $G_{2m}$ is an even function, we may write
\[
c_{2m}
=\omega_{d-2}\int_0^1 G_{2m}(t)\, (1-t^2)^{\frac{d-3}{2}}\,2t\,dt.
\]
The $m=0$ case is easy: $G_0(t)= 1$, and so $c_0>0$.
Now suppose $m\geq1$.
Setting $\alpha:=\frac{d-2}{2}$, we apply the \textit{Rodrigues formula}:
\[
G_{2m}(t)\, (1-t^2)^{\alpha-\frac{1}{2}}
=\frac{\Gamma(\alpha+\frac{1}{2})}{4^m\Gamma(\alpha+2m+\frac{1}{2})}\cdot F^{(2m)}(t),
\qquad
F(t):=(1-t^2)^{2m+\alpha-\frac{1}{2}}.
\]
Then $c_{2m}$ is a positive multiple of the following integral, which succumbs to integration by parts:
\[
\int_0^1 t\,F^{(2m)}(t)\,dt
=F^{(2m-1)}(1)-F^{(2m-2)}(1)+F^{(2m-2)}(0).
\]
Since $d\geq 3$, the $k$th derivative of $F(t)$ contains a positive power of $1-t^2$ for every $k\leq 2m-1$.
In particular, $F^{(2m-1)}(1)=F^{(2m-2)}(1)=0$.
Next, defining $f(u):=(1-u)^{2m+\alpha-\frac{1}{2}}$, then
\[
F'(t)
=2t\,f'(t^2),
\qquad
F''(t)
=(2t)^2\,f''(t^2)+2f'(t^2).
\]
Continuing inductively, we then get
\[
F^{(2k)}(0)
=\frac{(2k)!}{k!}f^{(k)}(0)
\]
for every $k\leq m-1$.
Putting $k=m-1$ and $s:=2m+\alpha-\frac{1}{2}$, then $c_{2m}$ is a positive multiple of
\[
f^{(m-1)}(0)
=(-1)^{m-1}s(s-1)\cdots(s-m+2)
\neq0.
\]
That is, $c_{2m}\neq0$.
\end{proof}

\section{Containment of function spaces and proof of Theorem~\ref{thm.main result ii}}
\label{sec.proof of containments}

The purpose of this section is to prove Theorem~\ref{thm.main result ii}.
We prove each part of this result in a different subsection below.
One of the features we use to prove proper containment of function spaces is differentiability.
This is made possible thanks to the following (standard) result concerning the space $\operatorname{Lip}_0(\mathbb{R}^d)$ of Lipschitz functions $\mathbb{R}^d\to\mathbb{R}$ that vanish at the origin:

\begin{proposition}
\label{prop.C^1 is Lipschitz closed}
Suppose $\mathcal{F}\subseteq\operatorname{Lip}_{0}(\mathbb{R}^d)$ consists of functions that are continuously differentiable over an open set $U\subseteq\mathbb{R}^d$.
Then every member of $\overline{\mathcal{F}}$ is also continuously differentiable over $U$.
\end{proposition}

\begin{proof}[Proof sketch]
Fix $f\in\overline{\mathcal{F}}$.
Then there exists a sequence $\{f_n\}_{n=1}^\infty$ in $\mathcal{F}$ such that $\|f_n-f\|_{\operatorname{Lip}}\to0$.
In particular, $\{f_n\}_{n=1}^\infty$ is Cauchy in $\operatorname{Lip}_{0}(\mathbb{R}^d)$.
Consider the inequality
\[
\sup_{x\in U}\|\nabla f_m(x)-\nabla f_n(x)\|
\leq \|f_m-f_n\|_{\operatorname{Lip}}
\qquad
\forall\,m,n\in\mathbb{N},
\]
which can be obtained either as a consequence of Proposition~\ref{prop.rademacher}, or by a direct argument involving the mean value theorem.
This inequality implies that $\{\nabla f_n|_U\}_{n=1}^\infty$ is Cauchy in the space of uniformly bounded functions $U\to\mathbb{R}^d$ with uniform norm, meaning $\{\nabla f_n|_U\}_{n=1}^\infty$ converges uniformly to some $g\colon U\to\mathbb{R}^d$.
Since the uniform limit of continuous functions is continuous, it holds that $g$ is continuous.
It remains to show that $f$ is differentiable at each $x\in U$ with $\nabla f(x)=g(x)$, which amounts to a standard exercise (see, for example, Theorem~7.17 in~\cite{Rudin:76} for a related argument).
\end{proof}

Next, we derive important structural properties of the space $\operatorname{PL}_{\operatorname{h}}$ of positively homogeneous piecewise linear functions $\mathbb{R}^d\to\mathbb{R}$.
(These properties are essentially known in the literature, but we reproduce them here for convenience.)
For every $f\in\operatorname{PL}_{\operatorname{h}}$, there is a finite \textit{tessellation} $\mathcal{T}$ of $\mathbb{R}^d$ (meaning $\mathcal{T}$ consists of subsets of $\mathbb{R}^d$ that cover $\mathbb{R}^d$ and have pairwise disjoint interiors, and each member of $\mathcal{T}$ equals the closure of its interior) such that the level sets of $\nabla f$ are precisely the interiors of the members of $\mathcal{T}$.
Since $f$ is positively homogeneous, each $C\in\mathcal{T}$ is closed under positive scalar multiplication, and furthermore, for each $C\in\mathcal{T}$, there exists $v_C\in\mathbb{R}^d$ (namely, the constant gradient of $f$ over $\operatorname{int} C$) such that $f(x)=\langle x,v_C\rangle$ for all $x\in\operatorname{int} C$.
Finally, since $f$ is continuous, the tuple $(v_C)_{C\in\mathcal{T}}\in(\mathbb{R}^d)^{\mathcal{T}}$ that determines $f$ necessarily satisfies the linear constraints
\[
\langle x,v_C\rangle=\langle x,v_{C'}\rangle
\qquad
\forall\, C,C'\in\mathcal{T},\,C\neq C',\,x\in C\cap C'.
\]
Given $C,C'\in\mathcal{T}$ with $C\neq C'$, it follows from the above constraints that $C\cap C'$ is contained in the hyperplane $\operatorname{span}\{v_C-v_{C'}\}^\perp$.
Since the sets in $\mathcal{T}$ are closed and cover $\mathbb{R}^d$ with disjoint interiors, it follows that every boundary point of every $C\in\mathcal{T}$ necessarily resides in some other $C'\in\mathcal{T}$.
Thus,
\[
\bigcup_{C\in\mathcal{T}}\partial C
=\bigcup_{\substack{C,C'\in\mathcal{T}\\C\neq C'}}C\cap C'
\subseteq \bigcup_{\substack{C,C'\in\mathcal{T}\\C\neq C'}}\operatorname{span}\{v_C-v_{C'}\}^\perp.
\]
In particular, the set complement of the union of hyperplanes on the right-hand side is a disjoint union of open chambers, and $f$ has constant gradient over each of these chambers.

Given a finite set $\mathcal{H}$ of linear hyperplanes, let $\operatorname{PL}_{\operatorname{h}}(\mathcal{H})$ denote the subspace of functions in $\operatorname{PL}_{\operatorname{h}}$ that have constant gradient over each open chamber induced by $\mathcal{H}$.
Then by the above reduction,
\begin{equation}
\label{eq.stratification of PL}
\operatorname{PL}_{\operatorname{h}}
=\bigcup_\mathcal{H}\operatorname{PL}_{\operatorname{h}}(\mathcal{H}),
\end{equation}
where the union is over all finite sets of linear hyperplanes.
Interestingly, these subspaces form a lattice since $\operatorname{PL}_{\operatorname{h}}(\mathcal{H})\leq\operatorname{PL}_{\operatorname{h}}(\mathcal{H}')$ is equivalent to $\mathcal{H}\subseteq\mathcal{H}'$.
If we let $\mathcal{T}(\mathcal{H})$ denote the tessellation of closed chambers induced by $\mathcal{H}$, then $\operatorname{PL}_{\operatorname{h}}(\mathcal{H})$ is finite dimensional on account of the injective linear map $\operatorname{grad}\colon\operatorname{PL}_{\operatorname{h}}(\mathcal{H})\to(\mathbb{R}^d)^{\mathcal{T}(\mathcal{H})}$ defined by $\operatorname{grad}f(C)=\nabla f(x)$ for any $x\in\operatorname{int}C$.
Define a graph with vertex set $\mathcal{T}(\mathcal{H})$ and adjacency $C\leftrightarrow C'$ whenever $C\cap C'$ has co-dimension $1$, and put
\[
\kappa(f)
:=\max_{C\leftrightarrow C'}\|\operatorname{grad}f(C)-\operatorname{grad}f(C')\|
\qquad
\forall \, f \in \operatorname{PL}_{\operatorname{h}}(\mathcal{H}).
\]
Since the graph is connected, it follows that $\kappa$ vanishes precisely on the subspace $(\mathbb{R}^d)^*\leq\operatorname{PL}_{\operatorname{h}}(\mathcal{H})$ of linear functionals and defines a norm on the quotient space $\operatorname{PL}_{\operatorname{h}}(\mathcal{H})/(\mathbb{R}^d)^*$.
In some sense, $\kappa(f)$ quantifies second derivative information about $f$.
Furthermore, the quantity $\kappa(f)$ is independent of the choice of $\mathcal{H}$ since viewing $f$ as a member of $\operatorname{PL}_{\operatorname{h}}(\mathcal{H}')$ for some $\mathcal{H}'\supseteq\mathcal{H}$ results in the same value of~$\kappa(f)$.
Having established this preliminary material, we now proceed to the proof of Theorem~\ref{thm.main result ii}.

\subsection{Proof of Theorem~\ref{thm.main result ii}(a): Planar rotations}

For this part of Theorem~\ref{thm.main result ii}, it suffices to assemble the theory that we developed in the previous sections.

\begin{proof}[Proof of Theorem~\ref{thm.main result ii}(a)]
Recall that $h\colon\mathbb{R}^2\to\mathbb{R}\times\mathbb{C}\cong\mathbb{R}^3$ is defined by $h(0)=0$ and 
\[
h(x)
=\|x\|\cdot\big(\,
\cos\tfrac{\pi}{2r},\,\,
(\tfrac{x_1+\mathrm{i}x_2}{\|x\|})^r \sin\tfrac{\pi}{2r}\,\big)
\]
for $x\neq0$.
Then for $x \in S^{1}$,
\[
h_1(x)
=\cos\tfrac{\pi}{2r},
\qquad
h_2(x)
=\operatorname{Re}[(x_1+\mathrm{i}x_2)^r ]\cdot\sin\tfrac{\pi}{2r},
\qquad
h_3(x)
=\operatorname{Im}[(x_1+\mathrm{i}x_2)^r ]\cdot\sin\tfrac{\pi}{2r}.
\]
Notably, these coincide with $G$-invariant polynomial functions $S^{1}\to\mathbb{R}$ of degree $0$, $r$, and $r$, respectively.
Thus, for each $i$, we have $h_i=\operatorname{ext}(h_i|_{S^{1}})\in\operatorname{ext}\mathcal{P}^G$, and so
\[
\operatorname{span}\{h_i\}_{i=1}^n
\leq \overline{\operatorname{ext}\mathcal{P}^G}.
\]
Furthermore, the containment is strict since one space is finite-dimensional and the other is not.

Next, consider any $p\in\operatorname{ext}\mathcal{P}^G$.
Then by Theorem~\ref{thm.integral combo for rotations of the plane}, there exists a polynomial $q$ such that 
\[
p(x)
=\int_{S^1} q(y) \, \llangle [x],[y] \rrangle \, d\omega(y)
\]
for all $x\in S^1$, and so Theorem~\ref{thm.integral combos are in lipschitz closure of linear of max filter} gives $p\in\overline{\operatorname{span}\mathcal{M}}$.
It follows that
\[
\overline{\operatorname{ext}\mathcal{P}^G}
\leq\overline{\operatorname{span}\mathcal{M}}.
\]
Furthermore, the containment is strict since every member of the left-hand side is continuously differentiable over $\mathbb{R}^2\setminus\{0\}$ by Proposition~\ref{prop.C^1 is Lipschitz closed}, while a nontrivial max filter is not (see Proposition~\ref{prop.grad max filter}).

Finally, every max filter is $G$-invariant, positively homogeneous, and piecewise linear.
Since these properties are closed under linear combinations, it follows that
\[
\operatorname{span}\mathcal{M}
\leq \operatorname{PL}_{\operatorname{h}}^G.
\]
To prove equality, we show that any $f\in\operatorname{PL}_{\operatorname{h}}^G$ can be expressed as a linear combination of max filters.
Fix a $G$-invariant collection $\mathcal{H}$ of hyperplanes (i.e., lines in this case) with $|\mathcal{H}|\geq2$ such that $f\in\operatorname{PL}_{\operatorname{h}}(\mathcal{H})^G$.
Then it suffices to show that $\operatorname{PL}_{\operatorname{h}}(\mathcal{H})^G$ is spanned by max filters.
To this end, let $U\subseteq S^1$ denote the set of all unit vectors contained in these hyperplanes.
We claim that each function in $\operatorname{PL}_{\operatorname{h}}(\mathcal{H})$ arises uniquely from a function $U \to \mathbb{R}$ by conical extension.
Indeed, if we consider the points from $U$ in counterclockwise order, the function is linear on the conic hull of adjacent points.
(We can take conic hulls here since $\mathcal{H}$ has size at least~$2$, and so each sector has angle strictly less than~$\pi$.)
Then the $G$-invariant members of $\operatorname{PL}_{\operatorname{h}}(\mathcal{H})$ similarly correspond to $G$-invariant functions $U \to \mathbb{R}$, and so $\operatorname{PL}_{\operatorname{h}}(\mathcal{H})^G$ has dimension $n:=|U|/|G|$.
Next, pick a set $S$ of $G$-orbit representatives of~$U$.
Then $|S|=|U|/|G|=n$.
Rotate each $s\in S$ counterclockwise by $\pi/|G|$ radians to get the template $z_s\in\mathbb{R}^2$.
Then the max filters $\{\llangle[\cdot],[z_s]\rrangle\}_{s\in S}$ reside in $\operatorname{PL}_{\operatorname{h}}(\mathcal{H})^G$.
Furthermore, since these functions have pairwise disjoint singularity sets (other than the origin), they are linearly independent.
As such, they span an $n$-dimensional subspace of $\operatorname{PL}_{\operatorname{h}}(\mathcal{H})^G$, i.e., all of $\operatorname{PL}_{\operatorname{h}}(\mathcal{H})^G$.
\end{proof}

\subsection{Proof of Theorem~\ref{thm.main result ii}(b): Real phase retrieval}

For this part of Theorem~\ref{thm.main result ii}, the material from the previous sections implies everything except that the final containment $\overline{\operatorname{span}\mathcal{M}}
\lneq \overline{\operatorname{PL}_{\operatorname{h}}^G}$ is proper.
In this section, we build up the theory necessary to prove this last piece before writing out the proof of Theorem~\ref{thm.main result ii}(b).

We are chiefly interested in the $G$-invariant functions in $\operatorname{PL}_{\operatorname{h}}$, namely, $\operatorname{PL}_{\operatorname{h}}^G$, and for our purposes, the most important feature of this function space is captured by the following lemma:

\begin{lemma}
\label{lem.ksl}
Take any $d>2$ and put $G := \{ \pm  I \} \leq \operatorname{O}(d)$.
Given $f\in\operatorname{PL}_{\operatorname{h}}^G$, denote
\[
K(f)
:=\{\,x \in \mathbb{R}^d : \text{$f$ is not differentiable at $x$}\,\}.
\]
If $K(f)$ is not a union of finitely many hyperplanes, then $f \notin \overline{\operatorname{span}\mathcal{M}}$.
\end{lemma}

\begin{figure}[t]
  \centering
\begin{tikzpicture}[scale=0.9,
  y={(4.6cm,-0.2cm)},             
  z={(0cm,4.6cm)},             
  x={(-0.8cm,-1.2cm)},          
  line cap=round, line join=round
]

\def\samp{120} 

\begin{scope}
  \colorlet{patchlight}{white}
  \colorlet{patchdark}{black!18}

  \clip
    plot[domain=0:90,samples=\samp,variable=\t] ({cos(\t)},{sin(\t)},0) -- 
    plot[domain=0:90,samples=\samp,variable=\t] (0,{cos(\t)},{sin(\t)}) -- 
    plot[domain=90:0,samples=\samp,variable=\t] ({cos(\t)},0,{sin(\t)}) -- 
    cycle;

  \shade[ball color=patchdark, opacity=0.5] (-1,0) circle [radius=3.2];

\end{scope}

\coordinate (X) at (1,0,0);  
\coordinate (Y) at (0,1,0);  
\coordinate (Z) at (0,0,1);  
\coordinate (O) at (0,0,0);  

\draw[dotted] (O) -- (X);
\draw[dotted] (O) -- (Y);
\draw[dotted] (O) -- (Z);

\coordinate (Mxy) at ({1/sqrt(2)},{1/sqrt(2)},0);
\coordinate (Mxz) at ({1/sqrt(2)},0,{1/sqrt(2)});
\coordinate (Myz) at (0,{1/sqrt(2)},{1/sqrt(2)});

\coordinate (D) at ({1/sqrt(3)},{1/sqrt(3)},{1/sqrt(3)});

\draw[very thick]
  plot[domain=0:90,samples=\samp,variable=\t]
    ({cos(\t)},{sin(\t)},0);          

\draw[very thick]
  plot[domain=0:90,samples=\samp,variable=\t]
    (0,{cos(\t)},{sin(\t)});           

\draw[very thick]
  plot[domain=0:90,samples=\samp,variable=\t]
    ({cos(\t)},0,{sin(\t)});           

\draw[thick]
  plot[domain=0:54.7356103,samples=\samp,variable=\t]
    ({sin(\t)/sqrt(2)},{sin(\t)/sqrt(2)},{cos(\t)});

\draw[thick]
  plot[domain=0:54.7356103,samples=\samp,variable=\t]
    ({sin(\t)/sqrt(2)},{cos(\t)},{sin(\t)/sqrt(2)});

\draw[thick]
  plot[domain=0:54.7356103,samples=\samp,variable=\t]
    ({cos(\t)},{sin(\t)/sqrt(2)},{sin(\t)/sqrt(2)});

\fill (D) circle[radius=1.1pt];

\node[right]      at (Y) {$x_2$}; 
\node[below left]      at (X) {$x_1$}; 
\node[above] at (Z) {$x_3$}; 

\node at ({1/sqrt(14)},{2.5/sqrt(14)},{2.5/sqrt(14)}) {$x_1$};

\node at ({2.5/sqrt(14)},{1/sqrt(14)},{2.5/sqrt(14)}) {$x_2$};

\node at ({2.5/sqrt(14)},{2.5/sqrt(14)},{1/sqrt(14)}) {$x_3$};
\end{tikzpicture}
\caption{Illustration of Example~\ref{ex.ksl} in the special case where $d=3$. Here, we display the value of $f(x)$ when $x$ is in the first octant of the unit sphere. For $x$ residing in most other octants, we have $f(x)=0$, though the function values in the octant opposite the first are determined by the fact that $f$ is even. The set of $x\in\mathbb{R}^3$ at which $f(x)$ is not differentiable consists of all scalar multiples of the points along the solid black curves.}
\label{fig:PL_example}
\end{figure}
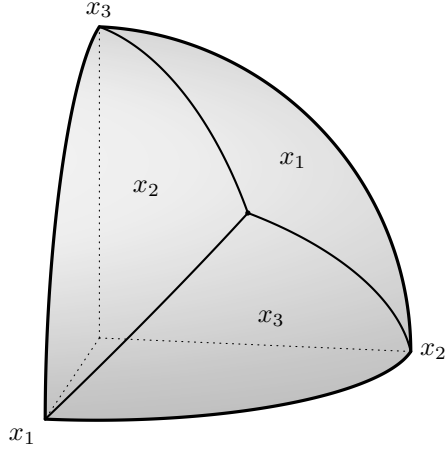

\begin{example}
\label{ex.ksl}
Take any $d>2$ and put $G := \{ \pm  I \} \leq \operatorname{O}(d)$.
Consider $f\in\operatorname{PL}_{\operatorname{h}}^G$ defined by
\[
f(x)
=\left\{\begin{array}{cl}\displaystyle\min_{i\in[d]}|x_i| &\text{if }\operatorname{sign}(x_1)=\cdots=\operatorname{sign}(x_d)\\[12 pt]
0&\text{else}.
\end{array}\right.
\]
See Figure~\ref{fig:PL_example} for an illustration.
Then $K(f)$ is a nonempty subset of the nonnegative and nonpositive orthants of $\mathbb{R}^d$, i.e., it is not a union of finitely many hyperplanes.
By Lemma~\ref{lem.ksl}, it follows that $f
\in\operatorname{PL}_{\operatorname{h}}^G\setminus\,\overline{\operatorname{span}\mathcal{M}}$.
\end{example}

\begin{proof}[Proof of Lemma~\ref{lem.ksl}]
First, we prove an intermediate claim, which is a special case of the contrapositive:
For every nonzero $f\in\operatorname{span}\mathcal{M}$, it holds that $K(f)$ is a union of finitely many hyperplanes.
Take any nonzero $f\in\operatorname{span}\mathcal{M}$, and consider any decomposition $f=\sum_{i=1}^n c_i|\langle\cdot,y_i\rangle|$ into max filters $\{|\langle\cdot,y_i\rangle|\}_{i=1}^n$ that are linearly independent as functions.
We claim that $K(f)=\bigcup_{i=1}^n\operatorname{span}\{y_i\}^\perp$.
For notational convenience, we use $U$ to denote the right-hand set.
The $\subseteq$ direction follows from the fact that any point of nondifferentiability in $f$ is a point of nondifferentiability for some max filter $|\langle\cdot,y_i\rangle|$.
For the $\supseteq$ direction, consider any $x\in U$.
Put $\delta_i:=\operatorname{sign}\langle x,y_i\rangle\in\{-1,0,1\}$.
Select a generic $v\in\mathbb{R}^d$.
Then for all $t\in\mathbb{R}$ with sufficiently small absolute value, it holds that
\begin{align*}
f(x+tv) 
= \sum_{i=1}^n c_i|\langle x+tv,y_i\rangle|
&= \sum_{\substack{i\in[n]\\[1 pt]y_i\not\perp x}} c_i|\langle x+tv,\delta_iy_i\rangle|+\sum_{\substack{i\in [n]\\[1 pt]y_i\perp x}} c_i|\langle x+tv,y_i\rangle|\\
&= \sum_{\substack{i\in[n]\\[1 pt]y_i\not\perp x}} c_i\langle x+tv,\delta_iy_i\rangle+\sum_{\substack{i\in[n]\\[1 pt]y_i\perp x}} c_i|\langle tv,y_i\rangle|\\
&=f(x)+t\sum_{\substack{i\in[n]\\[1 pt]y_i\not\perp x}} c_i\langle v,\delta_iy_i\rangle+|t|\sum_{\substack{i\in[n]\\[1 pt]y_i\perp x}} c_i|\langle v,y_i\rangle|.
\end{align*}
Since $x\in U$, there necessarily exists at least one $i$ for which $y_i\perp x$.
Furthermore, the coefficient of $|t|$ is nonzero since $v$ is generic.
It follows that $f$ is not differentiable at $x$ in the direction of $v$, and so $x\in K(f)$, as claimed.

We are now ready to prove the lemma.
Take any $f\in\operatorname{PL}_{\operatorname{h}}^G$ for which $K(f)$ is not a union of finitely many hyperplanes.
Since $f\in\operatorname{PL}_{\operatorname{h}}$, the stratification~\eqref{eq.stratification of PL} gives that there exists a set $\mathcal{H}$ of hyperplanes for which $f\in\operatorname{PL}_{\operatorname{h}}(\mathcal{H})$.
For each $H\in\mathcal{H}$, select a unit vector $z_H\in H^\perp$, and let $T\leq\operatorname{PL}_{\operatorname{h}}(\mathcal{H})$ denote the span of the resulting max filters $\{|\langle\cdot,z_H\rangle|\}_{H\in\mathcal{H}}$.
Note that for every linear functional $\ell\in(\mathbb{R}^d)^*$, the function $f+\ell\in\operatorname{PL}_{\operatorname{h}}(\mathcal{H})$ is nonzero with $K(f+\ell)=K(f)$.
Meanwhile, by our intermediate claim, every nonzero member of $T$ fails to be differentiable over an entire linear hyperplane in $\mathcal{H}$.
Since $f+(\mathbb{R}^d)^*$ avoids $T$, and since $\kappa$ defines a norm on $\operatorname{PL}_{\operatorname{h}}(\mathcal{H})/(\mathbb{R}^d)^*$, it follows that
\[
c
:=\inf_{g\in T}\kappa(g-f)
>0.
\]
Next, we show that for every linear combination $\sum_{i=1}^n c_ig_i$ of max filters $\{g_i\}_{i=1}^n$,
\begin{equation}
\label{eq.claimed bound}
\bigg\|\sum_{i=1}^n c_ig_i - f\bigg\|_{\operatorname{Lip}}
\geq\frac{c}{2},
\end{equation}
so that $f$ does not reside in the Lipschitz closure of the span of max filters.
Let $I\subseteq[n]$ denote the indices for which $g_i\in T$, and put
$g:=\sum_{i\in I} c_ig_i$
and $h:=\sum_{i\in[n]\setminus I} c_ig_i$.
Then $g-f\in\operatorname{PL}_{\operatorname{h}}(\mathcal{H})$ and $\kappa(g-f)\geq c$.
Furthermore, since $g_i\not\in T$ for every $i\in[n]\setminus I$, none of the linear hyperplanes of non-differentiability in $h$ coincide with any of the linear hyperplanes in $\mathcal{H}$.
It follows that we may select $x,y\in\mathbb{R}^d$ from the interiors of adjacent chambers in $\mathcal{T}(\mathcal{H})$ in such a way that
\[
\|\nabla (g-f)(x)-\nabla (g-f)(y)\|=\kappa(g-f)
\qquad
\text{and}
\qquad
\nabla h(x)=\nabla h(y).
\]
This then gives
\begin{align*}
c
\leq\kappa(g-f)
=\|\nabla (g-f)(x)-\nabla (g-f)(y)\|
&=\|\nabla h(x)-\nabla h(y) + \nabla (g-f)(x)-\nabla (g-f)(y)\|\\
&\leq\|\nabla(h+g-f)(x)\|+\|\nabla(h+g-f)(y)\|\\
&\leq 2\|h+g-f\|_{\operatorname{Lip}},
\end{align*}
which rearranges to the claimed bound~\eqref{eq.claimed bound}.
\end{proof}

\begin{proof}[Proof of Theorem~\ref{thm.main result ii}(b)]
Recall that $h\colon\mathbb{R}^d\to\mathbb{R}^{d\times d}_{\operatorname{sym}}\cong\mathbb{R}^{n}$ is defined in terms of the operator square root by
\[
h(x)
=\sqrt{xx^\top}.
\]
Then for every $i\in[n]$, there exist $j,k\in[d]$ such that for every $x \in S^{d-1}$,
\[
h_i(x)
=x_jx_k,
\]
which gives a $G$-invariant polynomial function $S^{d-1}\to\mathbb{R}$ of degree $2$.
It follows that
\[
\operatorname{span}\{h_i\}_{i=1}^n
\leq \overline{\operatorname{ext}\mathcal{P}^G},
\]
and the containment is strict since one space is finite-dimensional and the other is not. 

Next, every member of $\operatorname{ext}\mathcal{P}^G$ resides in $\overline{\operatorname{span}\mathcal{M}}$ by Theorems~\ref{thm.integral of max filters for pmI invariant polynomials} and~\ref{thm.integral combos are in lipschitz closure of linear of max filter}, and so
\[
\overline{\operatorname{ext}\mathcal{P}^G}
\leq\overline{\operatorname{span}\mathcal{M}}.
\]
Furthermore, the containment is strict since every member of the left-hand side is continuously differentiable over $\mathbb{R}^d\setminus\{0\}$ by Proposition~\ref{prop.C^1 is Lipschitz closed}, while a nontrivial max filter is not (see Proposition~\ref{prop.grad max filter}).

Finally, $\mathcal{M}\subseteq\operatorname{PL}_{\operatorname{h}}^G$ implies
\[
\overline{\operatorname{span}\mathcal{M}}
\leq \overline{\operatorname{PL}_{\operatorname{h}}^G},
\]
and the containment is strict as a consequence of Lemma~\ref{lem.ksl} and Example~\ref{ex.ksl}.
\end{proof}

\subsection{Proof of Theorem~\ref{thm.main result ii}(c): Reflection groups}

This final part of Theorem~\ref{thm.main result ii} is qualitatively different from the previous parts, as our analysis does not factor through the ideas developed in Section~\ref{sec.integral}.
Instead, it relies on the distinguishing characteristic that $\operatorname{span}\mathcal{M}$ has finite dimension.

\begin{proof}[Proof of Theorem~\ref{thm.main result ii}(c)]
Recall that $h\colon\mathbb{R}^d\to\mathbb{R}^d$ is defined by
\[
h(x)
\in [x]\cap W.
\]
Consider the map $L\colon\mathbb{R}^d\to\operatorname{Lip}_{\operatorname{h}}(\mathbb{R}^d)$ defined by
\[
L(y)(x)
=\langle h(x),y\rangle
\qquad
(x\in\mathbb{R}^d).
\]
Then $L$ is linear, and since $h_i=L(e_i)$ for each $i\in[d]$, it follows that $\operatorname{im}L=\operatorname{span}\{h_i\}_{i=1}^d$.
Next, for each $z\in\mathbb{R}^d$, we apply equation~(2) in~\cite{MixonP:23} to obtain
\[
\llangle[\cdot],[z]\rrangle
=\langle h(\cdot),h(z)\rangle
=L(h(z)).
\]
Since $\operatorname{im}h=W$ spans $\mathbb{R}^d$, it follows that
\[
\operatorname{span}\mathcal{M}
=\operatorname{span}\{L(h(z)):z\in\mathbb{R}^d\}
=L(\operatorname{span} W)
=\operatorname{im}L
=\operatorname{span}\{h_i\}_{i=1}^d.
\]
This subspace equals its closure in $\operatorname{Lip}_0(\mathbb{R}^d)$ since it has finite dimension.
Since every function in this subspace is $G$-invariant, positively homogeneous, and piecewise linear, we also have
\[
\overline{\operatorname{span}\mathcal{M}}
\leq\overline{\operatorname{PL}_{\operatorname{h}}^G}.
\]
This containment is proper since $\operatorname{PL}_{\operatorname{h}}^G$ has infinite dimension, which in turn can be witnessed by an infinite sequence $\{f_n\}_{n=1}^\infty$ in $\operatorname{PL}_{\operatorname{h}}^G$ such that each $f_n$ is nondifferentiable at some point $x_n$ at which $f_1,\ldots,f_{n-1}$ are differentiable.
(For example, first define $f_n|_W$ in such a way that its points of nondifferentiability reside in a new hyperplane, and then extend $G$-invariantly.)

Finally, assume no nonzero vector is fixed by all of $G$, i.e., $G$ is essential.
We argue that $\overline{\operatorname{ext}\mathcal{P}^G}\cap\overline{\operatorname{span}\mathcal{M}}=\{0\}$.
By Proposition~\ref{prop.C^1 is Lipschitz closed}, every member of $\overline{\operatorname{ext}\mathcal{P}^G}$ is continuously differentiable on $\mathbb{R}^d\setminus\{0\}$.
Meanwhile, every nonzero $f\in\overline{\operatorname{span}\mathcal{M}}$ can be viewed as the $G$-invariant extension of $\langle\cdot,y\rangle|_W$ for some nonzero $y\in\mathbb{R}^d$.
We will use this representation to show that $f$ is not differentiable at some point in $\mathbb{R}^d\setminus\{0\}$.
Since $G$ is essential as a reflection group, it follows that $W$ is an intersection of half spaces of the form $\{x\in\mathbb{R}^d:\langle x,\alpha_i\rangle\geq0\}$ for $i\in[d]$, where $\{\alpha_i\}_{i=1}^d$ forms a basis for $\mathbb{R}^d$.
Since $y$ is nonzero, there exists $i\in[d]$ such that $\langle y,\alpha_i\rangle\neq0$.
We claim that $f$ is not differentiable on the face $W\cap \alpha_i^\perp$ of $W$.
To see this, let $s_i\in G$ denote reflection about $\alpha_i^\perp$.
Then for every $x\in \operatorname{int}(W)$, we have
\[
f(s_ix)
=f(x)
=\langle x,y\rangle
=\langle s_i^2x,y\rangle
=\langle s_ix,s_iy\rangle.
\]
Since we chose $i\in[d]$ so that $\langle y,\alpha_i\rangle\neq 0$, it follows that 
\[
\nabla f(s_ix)
=s_iy 
= y-2\frac{\langle y,\alpha_i\rangle}{\|\alpha_i\|^2}\alpha_i
\neq y
=\nabla f(x).
\]
Since the gradient of $f$ is different over $\operatorname{int}(W)$ and $\operatorname{int}(s_iW)$, the claim follows.
\end{proof}

\section{Numerical results}
\label{sec.numerical}

In the previous sections, we demonstrated that, at least for certain finite subgroups $G\leq\operatorname{O}(d)$, one can obtain a nearly optimal bilipschitz embedding of the orbit space $\mathbb{R}^d/G$ into Euclidean space by post-composing an appropriate max filter bank with a linear map.
In this section, we present numerical experiments that illustrate the dramatic extent to which this phenomenon appears to generalize.
We consider multiple settings in which we have 
\begin{itemize}
\item a real Hilbert space $V$ (typically $\mathbb{R}^d$), 
\item a group $G\leq\operatorname{O}(V)$,
\item a training set $\mathcal{X}=\{x_i\}_{i=1}^N$ in $V$ from distinct $G$-orbits, and
\item a target dimension $n$,
\end{itemize}
and our goal is to find a $G$-invariant map $f\colon V\to\mathbb{R}^n$ that minimizes the \textit{empirical distortion} $\operatorname{dist}_\mathcal{X}(f^\downarrow)$, defined by
\[
\operatorname{dist}_\mathcal{X}(f^\downarrow)
:=\frac{\beta_\mathcal{X}(f^\downarrow)}{\alpha_\mathcal{X}(f^\downarrow)},
\quad
\alpha_\mathcal{X}(f^\downarrow):=\min_{\substack{i,j\in[N]\\i\neq j}}\frac{\|f(x_i)-f(x_j)\|}{d_{V/G}([x_i],[x_j])},
\quad
\beta_\mathcal{X}(f^\downarrow):=\max_{\substack{i,j\in[N]\\i\neq j}}\frac{\|f(x_i)-f(x_j)\|}{d_{V/G}([x_i],[x_j])}.
\]
There are two types of (nearly) $G$-invariant maps that we consider: \textit{trained maps} (in which the map belongs to some specified parameterized family, and the parameters are optimized for empirical distortion) and \textit{untrained maps} (in which the map is given by an explicit formula, though in some cases there is some randomness in the definition of the map).
In both cases, we report the empirical distortion $\operatorname{dist}_\mathcal{Y}(f^\downarrow)$ for some test set $\mathcal{Y}$ in $V$ that is independent of the training set $\mathcal{X}$.

We test a variety of architectures for training (nearly) $G$-invariant maps:
\begin{itemize}
\item
\textbf{Trained max filter bank (MF).}
Learn unit-norm templates of a max filter bank.
\item
\textbf{Trained linear layer $\circ$ random max filter bank (LRMF).}
Draw a random max filter bank~$\Phi$ whose templates have standard Gaussian entries, and then train a linear map $L$ so that the resulting map is $L \circ \Phi$.
\item
\textbf{Trained linear layer $\circ$ max filter bank (LMF).}
Train a max filter bank $\Phi$ simultaneously with a linear layer $L$ so that the resulting map is $L \circ \Phi$.
\item
\textbf{Trained ReLU neural network (ReLU)}.
Train two linear maps $W\colon V\to\mathbb{R}^m$ and $L\colon\mathbb{R}^m\to\mathbb{R}^n$ so that the resulting map is $L\circ\operatorname{ReLU}\circ W$.
Notably, this map is not $G$-invariant.
We train this map using the $G$-invariant augmentation $\{gx_i\}_{i\in[n],g\in G}$ of the training set $\mathcal{X}$, though when $G$ is infinite, we randomly sample $G$ in our augmentation.
\end{itemize}
We also consider the following untrained $G$-invariant maps:
\begin{itemize}
\item
\textbf{Random max filter bank (RMF).}
Draw a random max filter bank whose templates have standard Gaussian entries.
\item
\textbf{$G$-invariant polynomial (Poly).}
This is a particular $G$-invariant polynomial given by the $n$ real coordinates of the map $p$ specified in Table~\ref{table.G inv polys}.
(In the case of $\ell^2(\mathbb{Z}_d)$, we denote by $\hat{x}$ the discrete Fourier transform of $x$.)
\item
\textbf{Homogeneous extension of $G$-invariant polynomial (HPoly).}
For the corresponding $G$-invariant polynomial described in Table~\ref{table.G inv polys}, we restrict to the unit sphere and take the homogeneous extension to obtain a positively homogeneous map.
\end{itemize}

\begin{table}[t]
\caption{Group-invariant polynomial maps\label{table.G inv polys}}
\centering
\begin{tabular}{>{$}c<{$} >{$}c<{$} >{$}c<{$} >{$}c<{$} >{$}c<{$}}\Xhline{1.2pt}
\noalign{\vskip 0.3em}
V & G & p(x) & \text{native target space} & n \\[0.3em] \Xhline{1.2pt}
\noalign{\vskip 0.3em}
\mathbb{R}^d & \{\pm I\} & xx^\top & \mathbb{R}^{d\times d}_{\operatorname{sym}} & \binom{d+1}{2} \\[0.3em]\hline
\noalign{\vskip 0.3em}
\mathbb{R}^d & S_d & \{\sum_{j=1}^d x_j^k\}_{k=1}^d & \mathbb{R}^d & d \\[0.3em]\hline
\noalign{\vskip 0.3em}
\ell^2(\mathbb{Z}_d) & C_d & \{\hat{x}(a)\hat{x}(b)\hat{x}(-a-b)\}_{a,b\in\mathbb{Z}_d} & \mathbb{C}^{d^2} & 2d^2 \\[0.3em]\hline
\noalign{\vskip 0.3em}
\mathbb{C} & C_r & x^r & \mathbb{C} & 2 \\[0.3em]\hline
\noalign{\vskip 0.3em}
\mathbb{C}^d & S^1 & xx^* & \mathbb{C}^{d\times d}_{\operatorname{sa}} & d^2 \\[0.3em]\hline
\noalign{\vskip 0.3em}
(\mathbb{R}^d)^k & \operatorname{O}(d) & \{\langle x_i,x_j\rangle\}_{i,j\in[k]} & \mathbb{R}^{k\times k}_{\operatorname{sym}} & \binom{k+1}{2}\\[0.3em]\Xhline{1.2pt}
\end{tabular}
\end{table}

In each experiment, we fix the final embedding dimension to be $n = 3 \operatorname{dim}(V/G)$, which is exactly $3d$ when $V = \mathbb{R}^d$ and $G$ is finite.
Our training set consists of vectors with independent standard Gaussian entries.
We trained $10$ times with different initializations.
For RMF, the minimum test set empirical distortion over $2000$ random max-filter banks is recorded. 
We report the best empirical distortions we achieved with a common test set of size $2000$ generated independently with the same distribution.
See Table~\ref{table:results} for the results.
In each row, we bold the trained maps with the smallest distortion, since we seek a universal approach to train low-distortion embeddings.
In some cases, we already know that HPoly is optimal, and in these cases, we italicize the distortion.
For cases of the form $\mathbb{R}^2/C_r$, we note that HPoly fails to achieve the optimal distortion because we selected straightforward invariants in Table~\ref{table.G inv polys}, while the optimal invariants in Example~\ref{ex.optimal coordinates} represent the homogeneous extension of a very particular choice of polynomials.
One distortion value is omitted from Table~\ref{table:results} because our data augmentation approach was too costly to consider ReLU for $S_5$.
In all settings, the best trained $G$-invariant map is LMF.
In the case of $\mathbb{R}^{3}/\{\pm I\}$, LMF is slightly outperformed by the approximate $G$-invariant map ReLU, perhaps because sign invariance is well aligned with the ReLU activation.

{\renewcommand{\arraystretch}{1.5}%
\begin{table}[p]
\caption{Minimum empirical distortions of different orbit spaces}
\label{table:results}
\vspace{-6pt}
\centering
\begin{tabular}{|l||c|c|c|c||c|c|c|}
\hline
\textbf{}  & \textbf{MF} & \textbf{LRMF} & \textbf{LMF} & \textbf{ReLU} & \textbf{RMF} & \textbf{Poly} & \textbf{HPoly}\\
\hline
{$\mathbb{R}^{3}/\{\pm I\} \to \mathbb{R}^{9}$} 
   & 1.85 & 2.56 & 1.76 & \textbf{1.73} & 2.40 & 17.70 & \textit{1.41} \\
\hline
\textbf{$\mathbb{R}^5/S_{5} \to \mathbb{R}^{15}$} 
  & 3.06 & \textbf{1.00} & \textbf{1.00} & — & 5.70 & {4.70e3} & {46.3} \\
\hline
\textbf{$\ell^2(\mathbb{Z}_5)/C_5 \to \mathbb{R}^{15}$} 
   & 4.72 & 3.57 & \textbf{2.88} & 19.0 & 5.58 & {5.18{e}{2}} & {132}\\
\hline
\textbf{$\mathbb{C}/C_3 \to \mathbb{R}^{6}$} 
   & 2.93 & 2.08 & \textbf{1.60} & 2.58{e}{3} & 3.05  & 2.42{e}{4} & 3.00 \\
\hline
\textbf{$\mathbb{C}/C_4 \to \mathbb{R}^{6}$} 
   & 4.16 & 2.08 & \textbf{1.54} & 16.20 & 4.43 & 3.89{e}{6} & 4.00 \\
\hline
\textbf{$ \mathbb{C}^{2}/S^1 \to \mathbb{R}^{9}$} 
  & 2.21 & 1.84 & \textbf{1.58} & 9.71 &  2.59 & 22.37 & \textit{1.41} \\
\hline
\textbf{$(\mathbb{R}^{3})^{3}/\operatorname{O}(3) \to \mathbb{R}^{18}$} 
  & 2.80 & 1.65 & \textbf{1.41} & 2.35{e}{3} & 4.79 & 10.79 & 3.78 \\
\hline
\end{tabular}
\end{table}

\begin{table}[p]
\caption{Minimum empirical distortions of $\mathbb{R}^2 / \{\pm I\}$}
\label{table:Z2_scaling}
\vspace{-6pt}
\centering
\begin{tabular}{|l||c|c|c|c||c|c|c|}
\hline
\textbf{} & \textbf{MF} & \textbf{LRMF} & \textbf{LMF} & \textbf{ReLU} & \textbf{RMF} & \textbf{Poly} & \textbf{HPoly} \\
\hline
\textbf{$\mathbb{R}^2/\{\pm I\} \to \mathbb{R}^8$} 
  & 1.84 & 2.01 & 1.66 & \textbf{1.64} & 1.92 & 163 & \textit{1.41}\\
\hline
\textbf{$\mathbb{R}^2/\{\pm I\} \to \mathbb{R}^{16}$} 
  & 1.67 & 1.54 & \textbf{1.46} & 1.63 & 1.78 & 163 & \textit{1.41}\\
\hline
\textbf{$\mathbb{R}^2/\{\pm I\} \to \mathbb{R}^{32}$} 
  & 1.66 &  1.49 & \textbf{1.44} & 1.58 & 1.80 & 163 & \textit{1.41}\\
\hline
\textbf{$\mathbb{R}^2/\{\pm I\} \to \mathbb{R}^{256}$} 
  & 1.66 & \textbf{1.42} & \textbf{1.42} & 1.51 & 1.76 & 163 & \textit{1.41}\\
\hline
\end{tabular}
\end{table}

\begin{table}[p]
\caption{Minimum empirical distortions of $\mathbb{C}^2 / S^1$}
\label{table:C2_scaling}
\vspace{-6pt}
\centering
\begin{tabular}{|l||c|c|c|c||c|c|c|}
\hline
\textbf{} & \textbf{MF} & \textbf{LRMF} & \textbf{LMF} & \textbf{ReLU} & \textbf{RMF} & \textbf{Poly} & \textbf{HPoly} \\
\hline
\textbf{$\mathbb{C}^2/S^1 \to \mathbb{R}^8$} 
  & 2.20 & 1.81 & \textbf{1.64} & 15.26 & 2.52 & 16.89 & \textit{1.41} \\
\hline
\textbf{$\mathbb{C}^2/S^1 \to \mathbb{R}^{16}$} 
  & 2.16  & 1.66 & \textbf{1.47} & 7.19 & 2.46 & 16.89 & \textit{1.41} \\
\hline
\textbf{$\mathbb{C}^2/S^1 \to \mathbb{R}^{32}$} 
  & 2.16  & 1.52 & \textbf{1.44} & 5.61 & 2.30 & 16.89 & \textit{1.41} \\
\hline
\textbf{$\mathbb{C}^2/S^1 \to \mathbb{R}^{256}$} 
  & 2.16  & \textbf{1.42} & \textbf{1.42} & 3.29 & 2.18 & 16.89 & \textit{1.41} \\
\hline
\end{tabular}
\end{table}

\begin{table}[p]
\caption{Minimum empirical distortions of  $(\mathbb{R}^{2})^2 / \operatorname{O}(2)$}
\label{table:O2_scaling}
\vspace{-6pt}
\centering
\begin{tabular}{|l||c|c|c|c||c|c|c|}
\hline
\textbf{} & \textbf{MF} & \textbf{LRMF} & \textbf{LMF} & \textbf{ReLU} & \textbf{RMF} & \textbf{Poly} & \textbf{HPoly} \\
\hline
\textbf{$(\mathbb{R}^{2})^2 / \operatorname{O}(2) \to \mathbb{R}^4$} 
  & 3.23  & 2.15 & \textbf{1.41} & 6.32e{4} & 3.47 & 47.63 & 8.75 \\
\hline
\textbf{$(\mathbb{R}^{2})^2 / \operatorname{O}(2) \to \mathbb{R}^8$} 
  & 2.55  & 1.46 & \textbf{1.40} & 5.23e{4} & 2.91 & 47.63 & 8.75 \\
\hline
\textbf{$(\mathbb{R}^{2})^2 / \operatorname{O}(2) \to \mathbb{R}^{16}$} 
  & 2.43 & \textbf{1.40} & \textbf{1.40} & 1.13e{5} & 2.77 & 47.63 & 8.75 \\
\hline
\end{tabular}
\end{table}

\begin{table}[p]
\caption{Minimum distortions of connected U.S.\ congressional districts}
    \label{table:district_dist}
    \vspace{-6pt}
    \centering
    \begin{tabular}{|l||c|c|c|c||c|}
        \hline
        & \textbf{MF} & \textbf{LRMF} & \textbf{LMF} & \textbf{ReLU} & \textbf{RMF} \\
        \hline
        \textbf{$(\mathbb{R}^{2})^{50}/(\operatorname{O}(2) \times C_{50}) \to \mathbb{R}^{100}$} 
        & 3.73 & 1.81 & \textbf{1.50} & 10.51 & 5.39 \\
        \hline
    \end{tabular}
\end{table}

\begin{table}[p]
    \caption{Minimum distortions of 2D Shape Structure Dataset}
    \label{table:shapes_dist}
    \vspace{-6pt}
    \centering
    \begin{tabular}{|l||c|c|c|c||c|}
        \hline
        & \textbf{MF} & \textbf{LRMF} & \textbf{LMF} & \textbf{ReLU} & \textbf{RMF} \\
        \hline
        \textbf{$(\mathbb{R}^{2})^{100}/(\operatorname{O}(2) \times C_{100}) \to \mathbb{R}^{200}$} 
        & 4.62 & 2.16 & \textbf{2.00} & 491 & 6.25 \\
        \hline
    \end{tabular}
\end{table}

Next, we isolate the cases of $\mathbb{R}^2/\{\pm I\}$, $\mathbb{C}^2/S^1$, and $(\mathbb{R}^2)^2/\operatorname{O}(2)$.
In the first case, the optimal distortion is known to be $\sqrt{2}$, and furthermore, this can be approached by LMF as a consequence of Theorem~\ref{thm.main result}(b).
This is corroborated by Table~\ref{table:Z2_scaling}.
Recall that the infimal distortion achievable by max filter banks is slightly larger: $\sqrt{\pi/(\pi-2)}\approx 1.659$.
The difference between MF and LMF in this case is illustrated in Figure~\ref{fig:R2Z2_viz}.
Apparently, the final linear layer in LMF has the effect of making the ice cream cone more squat.
For the other two cases where $\mathbb{C}^2/S^1$ and $(\mathbb{R}^2)^2/\operatorname{O}(2)$, we also know that the optimal distortion is~$\sqrt{2}$~\cite{CahillIM:24,Blum-SmitEtal:25}, but we do not yet have theory for LMF in this setting.
But judging from Tables~\ref{table:C2_scaling} and~\ref{table:O2_scaling}, LMF appears to approach the optimal distortion in these cases, too!
(The fact that the empirical distortion dips below $\sqrt{2}$ in the latter case stems from the fact that the empirical distortion quantifies the distortion of the test set, not the entire orbit space.)
Clearly, these two orbit spaces warrant further investigation as progress on Problem~\ref{problem.first problem}.

\begin{figure}[t]
    \centering

    \begin{subfigure}[b]{0.22\textwidth}
        \centering
        \includegraphics[height=\textwidth,trim=5.5cm 5cm 5cm 5cm,clip]{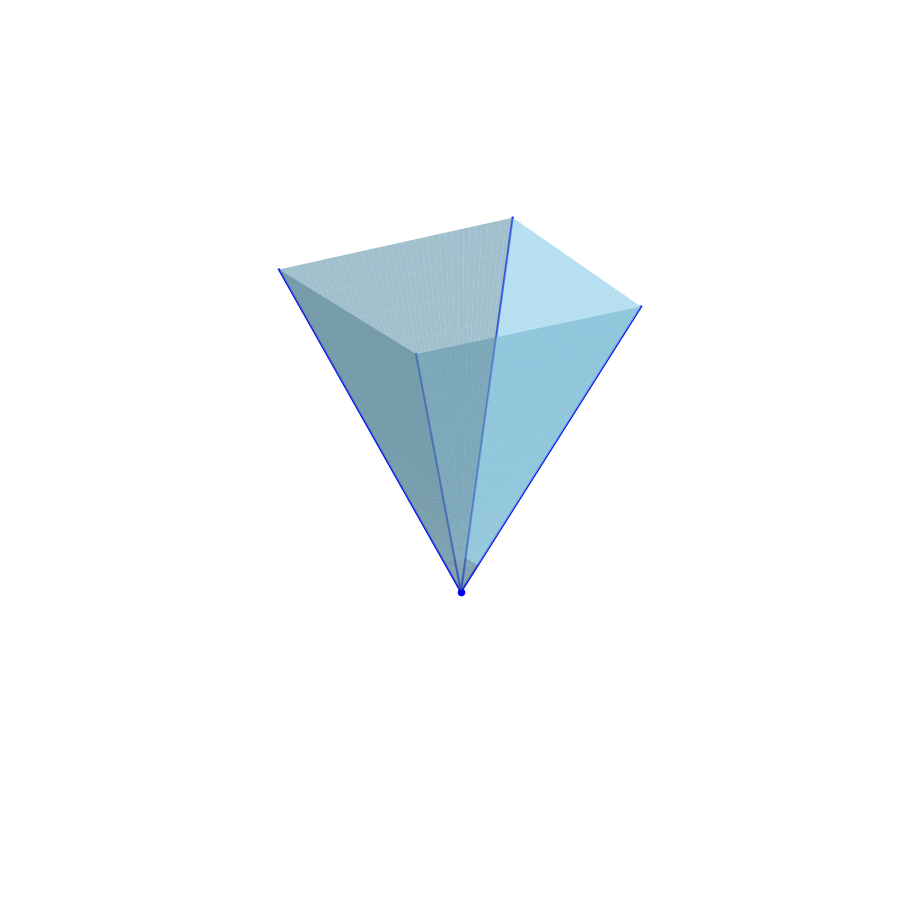}
        \subcaption{$n=4$}
    \end{subfigure}
    \hfill
    \begin{subfigure}[b]{0.22\textwidth}
        \centering
        \includegraphics[height=\textwidth,trim=5.5cm 5cm 5cm 5cm,clip]{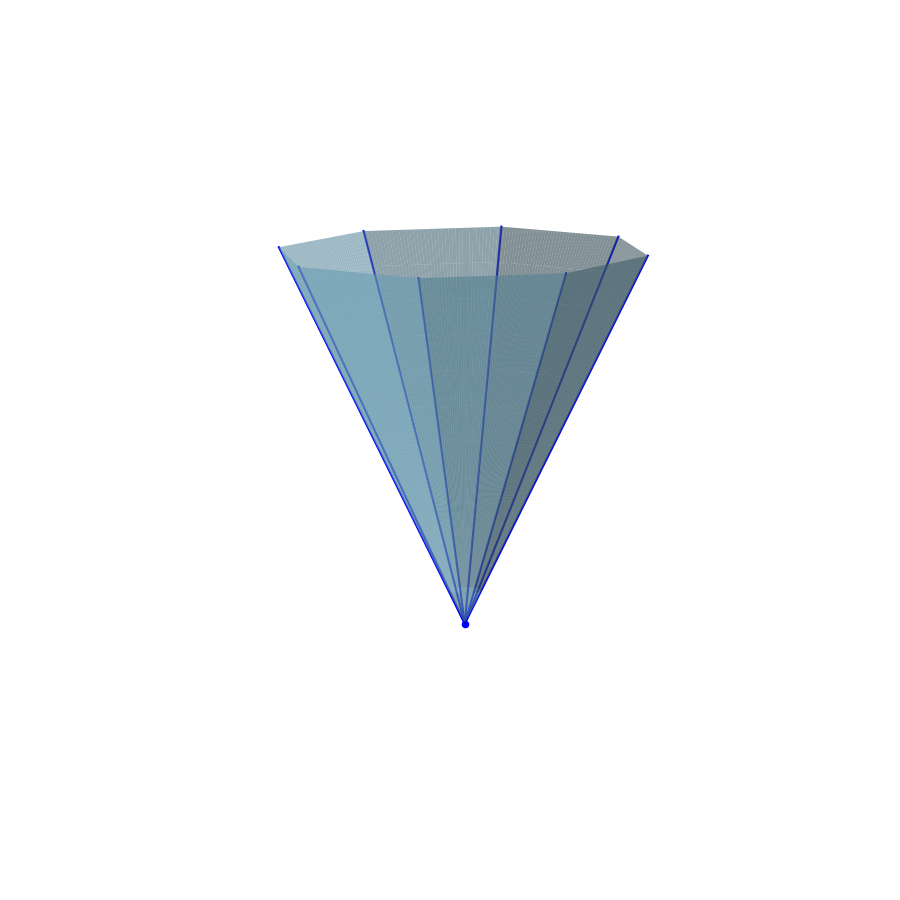}
        \subcaption{$n=8$}
    \end{subfigure}
    \hfill
    \begin{subfigure}[b]{0.22\textwidth}
        \centering
        \includegraphics[height=\textwidth,trim=6cm 5cm 5cm 5cm,clip]{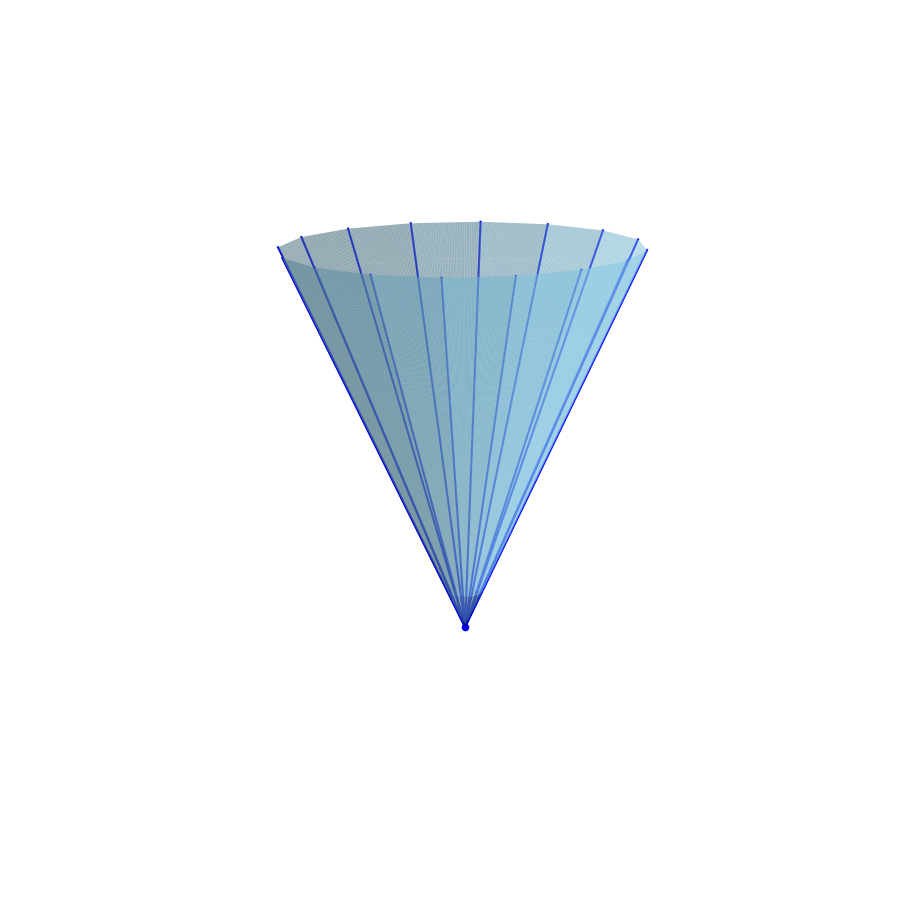}
        \subcaption{$n=16$}
    \end{subfigure}
    \hfill
    \begin{subfigure}[b]{0.22\textwidth}
        \centering
        \includegraphics[height=\textwidth,trim=5.7cm 5cm 5cm 5cm,clip]{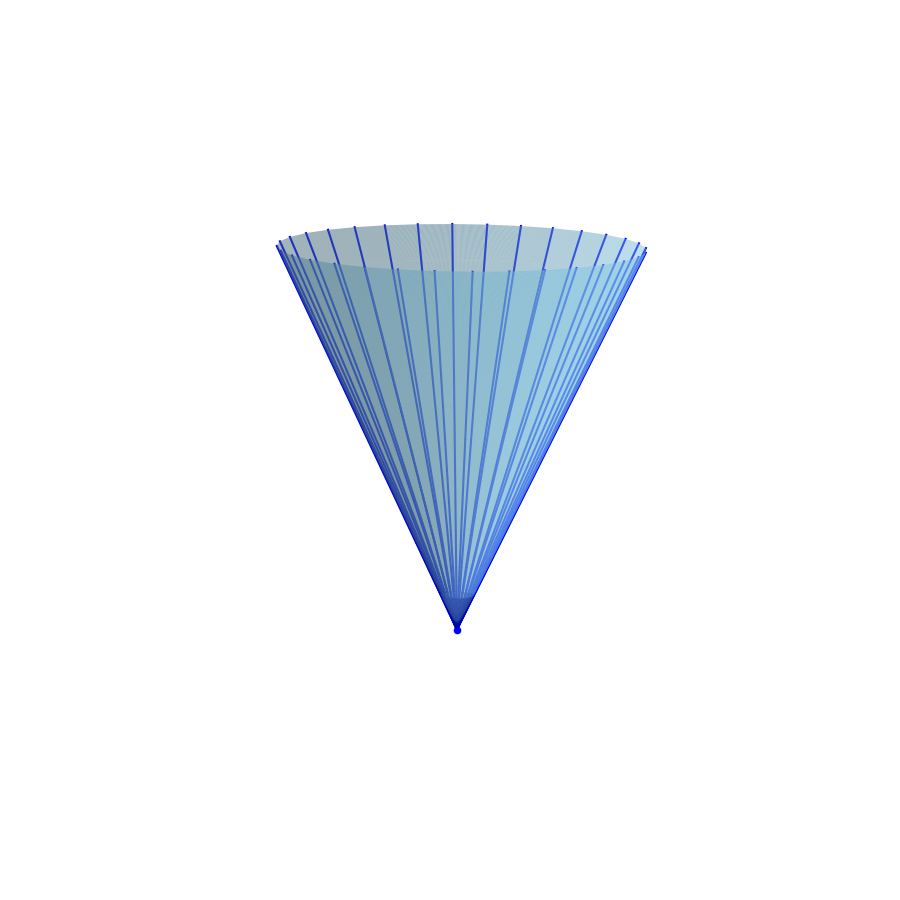}
        \subcaption{$n=32$}
    \end{subfigure}
    
    \vspace{1em}
    
    \begin{subfigure}[b]{0.22\textwidth}
        \centering
        \includegraphics[height=\textwidth,trim=5.5cm 5cm 5cm 5cm,clip]{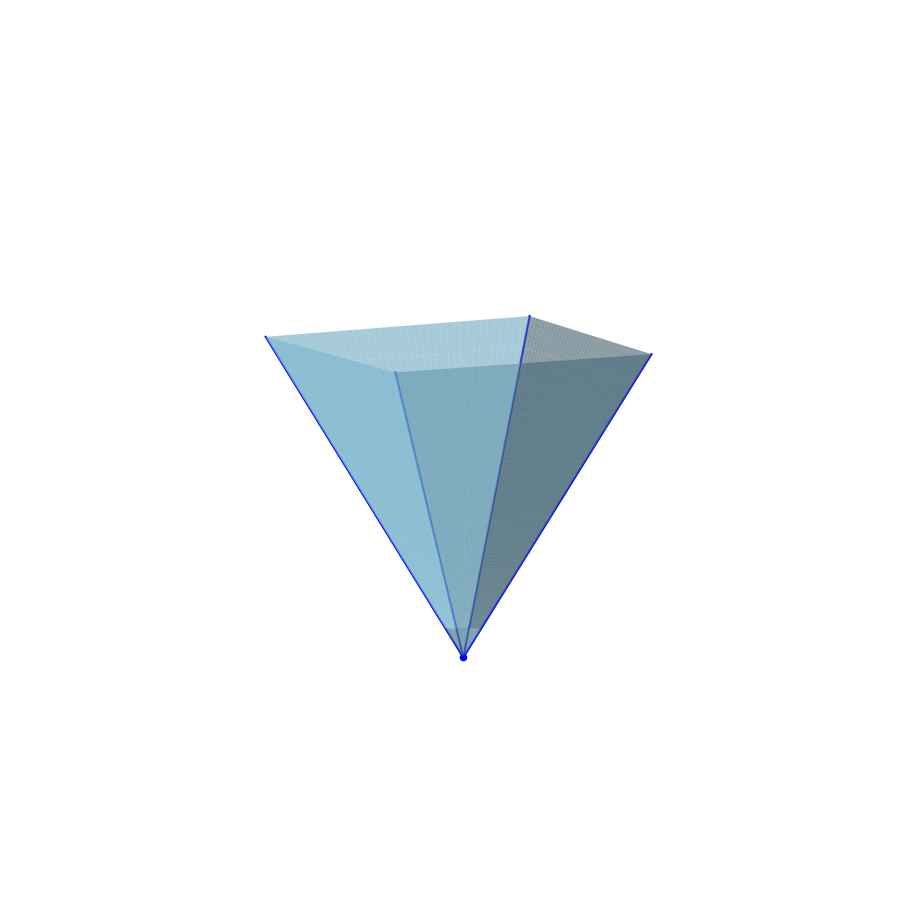}
        \subcaption{$n=4$}
    \end{subfigure}
    \hfill
    \begin{subfigure}[b]{0.22\textwidth}
        \centering
        \includegraphics[height=\textwidth,trim=4.2cm 5cm 5cm 5cm,clip]{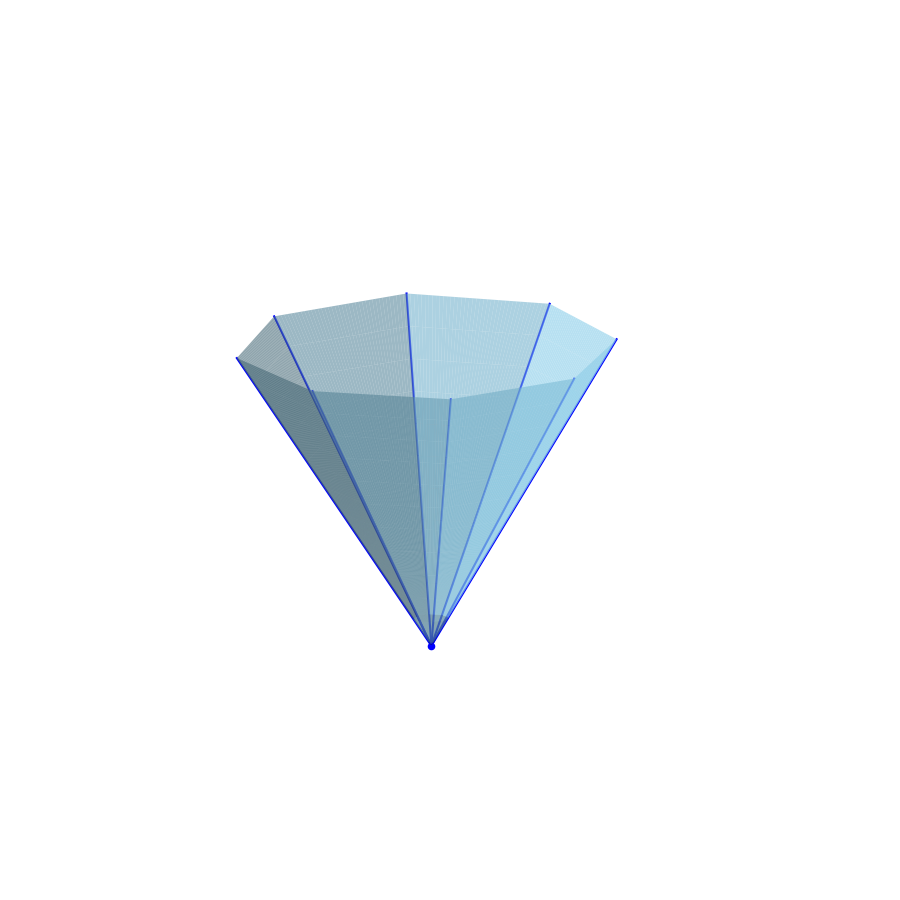}
        \subcaption{$n=8$}
    \end{subfigure}
    \hfill
    \begin{subfigure}[b]{0.22\textwidth}
        \centering
        \includegraphics[height=\textwidth,trim=6.9cm 5cm 5cm 5cm,clip]{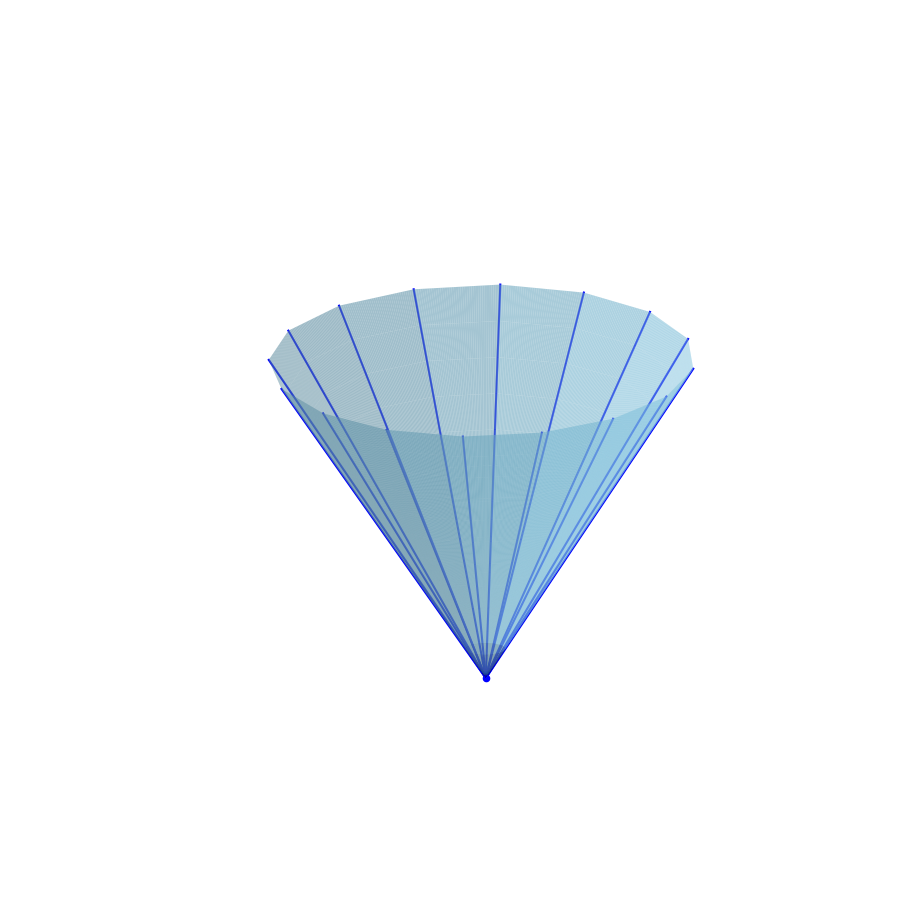}
        \subcaption{$n=16$}
    \end{subfigure}
    \hfill
    \begin{subfigure}[b]{0.22\textwidth}
        \centering
        \includegraphics[height=\textwidth,trim=4.6cm 5cm 5cm 5cm,clip]{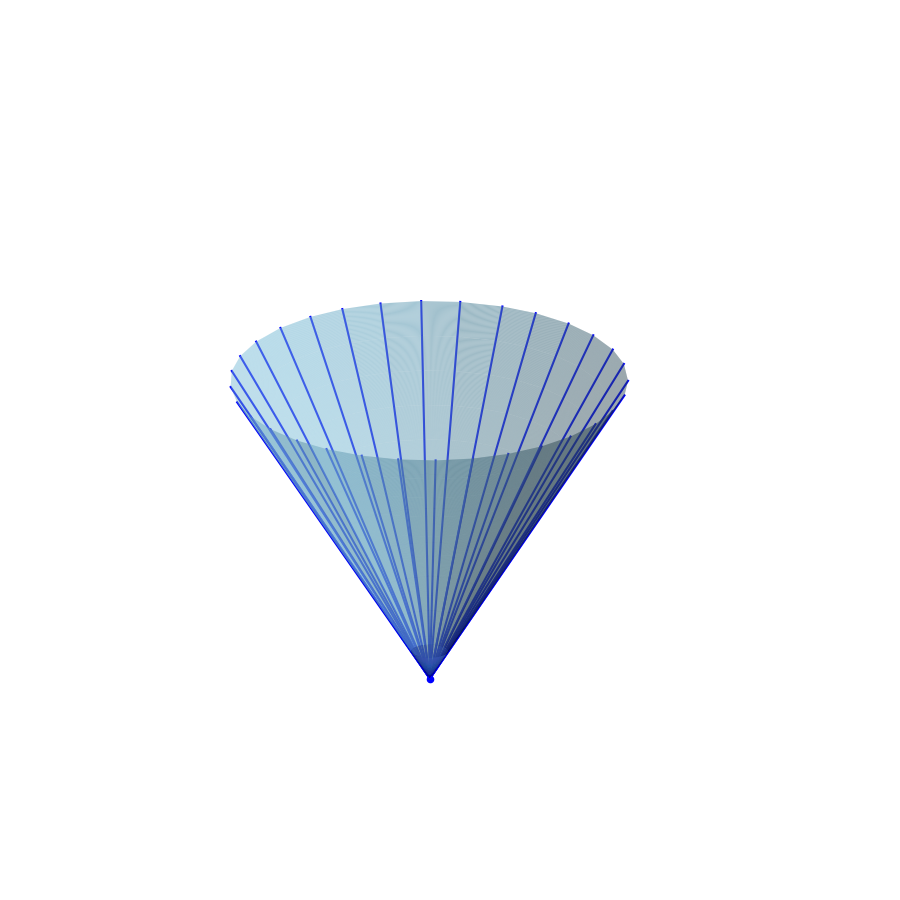}
        \subcaption{$n=32$}
    \end{subfigure}
    
    \caption{The embedding of $\mathbb{R}^2/\{\pm I\}$ into increasing embedding dimensions, visualized in $\mathbb{R}^3$ with the help of principal component analysis. The MF embeddings are above, and the LMF embeddings are below. Both embeddings use the optimal max filter bank, namely, roots of unity.}
    \label{fig:R2Z2_viz}
\end{figure}

In practice, we might not require a bilipschitz embedding of the entire orbit space, but rather of a subset $\mathcal{S}$ of the orbit space that data resides in.
This suggests that we use a training set (assumed to be representative of $\mathcal{S}$) to learn a bilipschitz embedding $\mathcal{S}\to\mathbb{R}^n$.
In what follows, we embed two different group-invariant training sets into a Euclidean space with low distortion. 
Here, we are not evaluating generalization, so we do not consider a test set that is distinct from the training set.
In both experiments, we consider shape data in the plane. 
We represent each example as a $k$-vertex origin-centered polygon, which we view as a member of the orbit space $(\mathbb{R}^{2})^{k}/(\operatorname{O}(2) \times C_{k})$.
Here, $\operatorname{O}(2)$ acts on the shape as a subset of the plane, while $C_k$ acts by cyclically permuting the vertices. 

Our first dataset consists of the 119th U.S.\ congressional districts. 
We discard the $17$ congressional districts that are disconnected, leaving $427$ districts to be embedded.
After sampling the boundary of each district at $k=50$ uniformly spaced points, we embed these in $\mathbb{R}^n$ with $n=100$.
Table~\ref{table:district_dist} reports the minimum distortions achieved over $10$ independently trained models of each type (or in the case of RMF, taking the minimum over $2000$ random max filter banks).
Figure \ref{fig:district_embedding} illustrates the top two principal components of the best LMF embedding. 
Apparently, the horizontal axis captures the extent to which a district is convex, while the vertical axis represents eccentricity.
The districts in red have been criticized by news outlets as being gerrymandered~\cite{Buchholz:online,Meyers:online}.
These districts generally appear on the bottom left, thereby suggesting that our embedding is geometrically relevant.

Next, we embed the 2D Shape Structure Dataset~\cite{CarlierEtal:online}.
This dataset consists of shapes from $70$ different classes. 
Each shape is sampled at $k=100$ points, and so we double the embedding dimension to $n=200$.
See Table~\ref{table:shapes_dist} for the distortion results. 
Once again, LMF achieves the best distortion. 
As before, we visualize this shape space in Figure~\ref{fig:shape_embedding}. 
Here, the horizontal axis appears to represent a measure of eccentricity, with longer shapes on the left and more compact ones on the right. 
Meanwhile, the vertical axis appears to represent a measure of convexity, though this is approximate as can be seen by the ``glass'' class that does not follow this trend.
We highlight a few shape classes for illustration purposes. 
Visual outliers in these clusters are embedded as outliers in shape space, such as a glass that does not resemble a typical glass or mugs with trivial topology. 

Both of these experiments illustrate the utility of the LMF approach to training a low-distortion map on group-invariant data.
These results suggest that group-invariant machine learning would benefit from training an LMF feature map.

\begin{figure}[p]
    \centering
    \includegraphics[width=\linewidth]{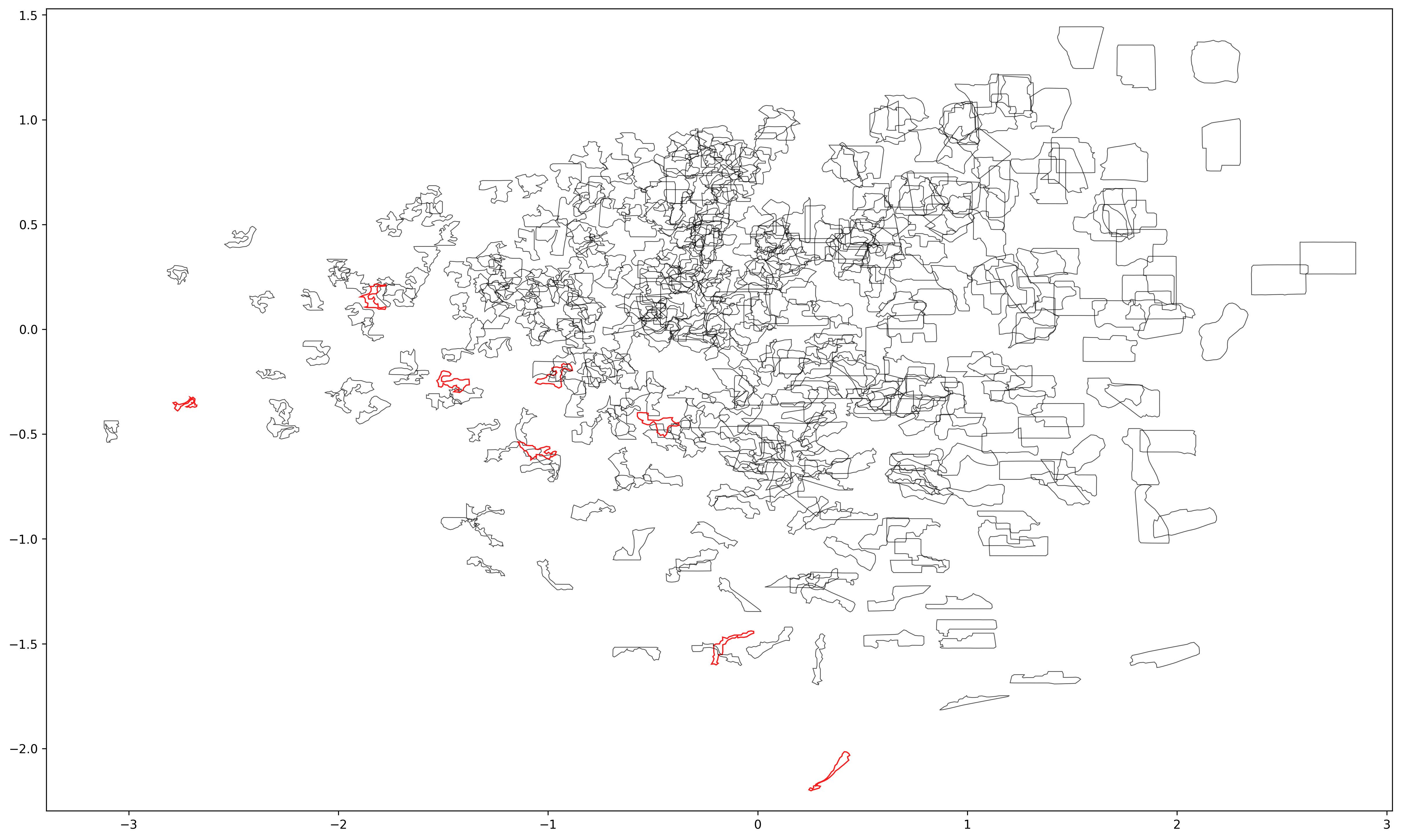}
    \caption{Connected U.S.\ congressional districts embedded using LMF, and then visualized using principal component analysis.}
    \label{fig:district_embedding}
\end{figure}

\begin{figure}[p]
    \centering
    \includegraphics[width=\linewidth]{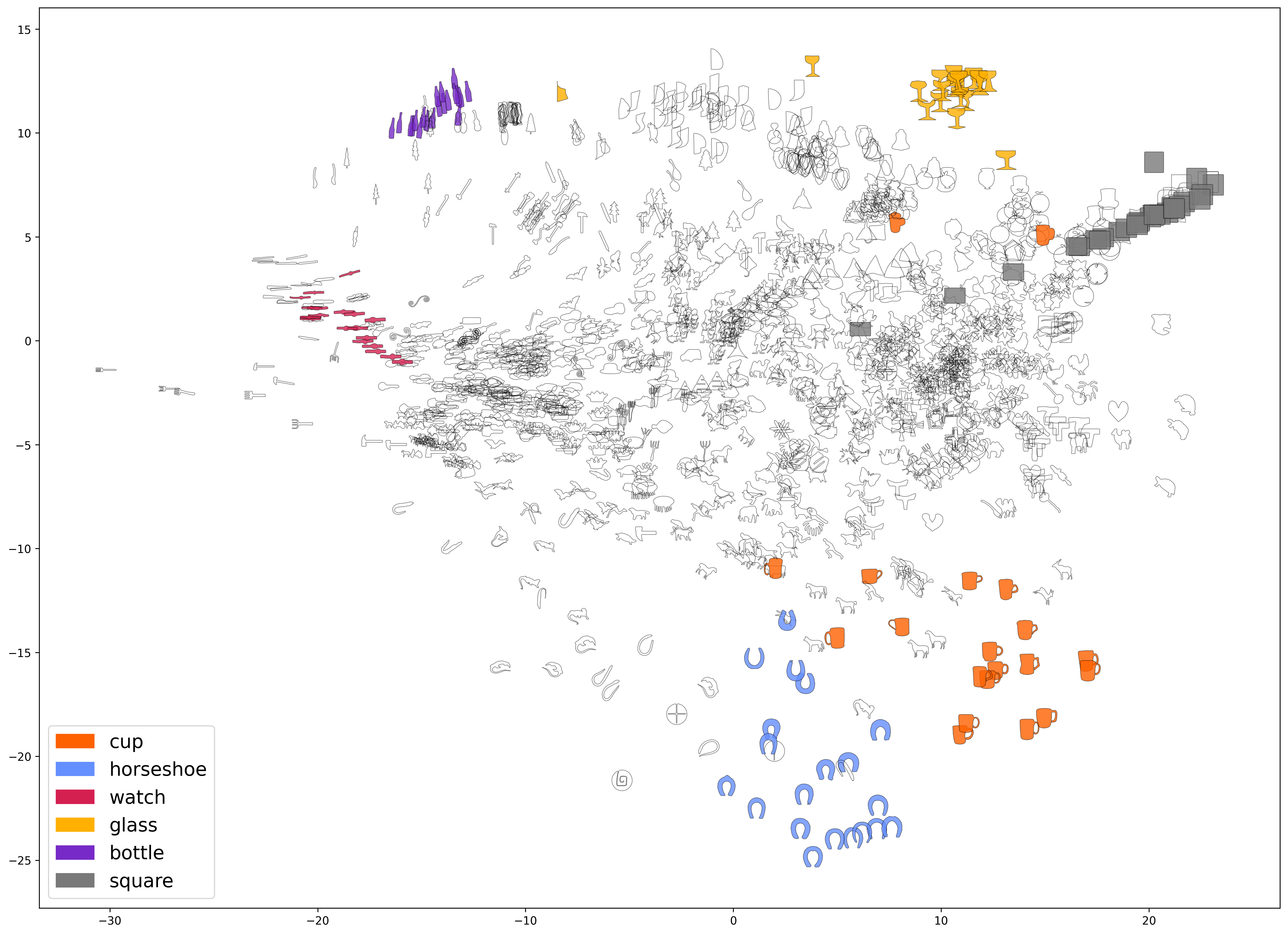}
    \caption{2D Shape Structure Dataset embedded using LMF, and then visualized using principal component analysis.}
    \label{fig:shape_embedding}
\end{figure}

\section*{Acknowledgments}

JWI was supported by a grant from the Simons Foundation, and by NSF DMS 2220301.
DGM was supported by NSF DMS 2220304, and thanks OpenAI for generously providing complimentary access to the Pro tier of ChatGPT, which suggested using Rodrigues formula in the proof of Theorem~\ref{thm.integral of max filters for pmI invariant polynomials}.

\end{document}